\documentclass[11pt,reqno,oneside]{amsart}
\usepackage{verbatim,amsmath,amssymb,cite,epsfig,xspace,color,subfigure}
\usepackage{pdfsync}

\numberwithin{equation}{section}

\theoremstyle{plain}
\newtheorem{theorem}{Theorem}[section]
\newtheorem{lemma}{Lemma}[section]
\newtheorem{corollary}{Corollary}[section]
\newtheorem{proposition}{Proposition}[section]
\theoremstyle{definition}

\theoremstyle{remark}
\newtheorem{remark}{Remark}[section]

\newcommand{\ignore}[1]{}

\newcommand{\yb}{\mathbf y}

\newcommand{\E}{{\mathcal E}}

\newcommand{\la}{\langle}
\newcommand{\ra}{\rangle}

\newcommand{\eps}{\epsilon}

\newcommand{\smfrac}[2]{{\textstyle \frac{#1}{#2}}}

\begin{document}

\title
[Positive-Definiteness of the Blended Force-Based Quasicontinuum Method]
{Positive-Definiteness of the Blended Force-Based Quasicontinuum Method}
\author{ Xingjie Helen Li, Mitchell Luskin and Christoph Ortner}


\thanks{ This work was supported in part by DMS-0757355, DMS-0811039,
  the PIRE Grant OISE-0967140, and the University of Minnesota
  Supercomputing Institute.  This work was also supported by the
  Department of Energy under Award Number DE-SC0002085. CO was
  supported by EPSRC Grant EP/H003096 ``Analysis of
  Atomistic-to-Continuum Coupling Methods''.}



\date{\today}

\begin{abstract}
  The development of consistent and stable quasicontinuum models for
  multi-dimensional crystalline solids remains a challenge. For
  example, proving stability of the force-based quasicontinuum (QCF)
  model \cite{Dobson:2008a} remains an open problem. In 1D and 2D, we
  show that by {\em blending} atomistic and Cauchy--Born continuum
  forces (instead of a sharp transition as in the QCF method) one
  obtains positive-definite blended force-based quasicontinuum (B-QCF)
  models. We establish sharp conditions on the required blending
  width.
\end{abstract}

\thispagestyle{empty}

\maketitle
\section{Introduction}

Atomistic-to-continuum coupling methods (a/c methods) have been
proposed to increase the computational efficiency of atomistic
computations involving the interaction between local crystal defects
with long-range elastic
fields~\cite{curt03,LinP:2006a,Miller:2003a,Shimokawa:2004,E:2004,Miller:2008,Legoll:2005,bqce11}.
Energy-based methods in this class, such as the quasicontinuum model
(denoted QCE \cite{Ortiz:1995a}) exhibit spurious interfacial forces
(``ghost forces'') even under uniform strain~\cite{Shenoy:1999a,
  Dobson:2008a}.  The effect of the ghost force on the error in
computing the deformation and the lattice stability by the QCE
approximation has been analyzed in~\cite{Dobson:2008a, Dobson:2008c,
  mingyang, doblusort:qce.stab}. The development of more accurate
energy-based a/c methods is an ongoing process
\cite{Shimokawa:2004,E:2004, Shapeev2D:2011,
  LuskinXingjie.qnl1d,OrtnerZhang:2011,XiaoBely:2002}.

An alternative approach to a/c coupling is the force-based
quasicontinuum (QCF) approximation~\cite{doblusort:qcf.stab, qcf.stab,
  curt03, Miller:2003a, Lu.bqcf:2011}, but the non-conservative and
indefinite equilibrium equations make the iterative solution and the
determination of lattice stability more
challenging~\cite{qcf.iterative, qcf.stab, DobShapOrt:2011}. Indeed,
it is an open problem whether the (sharp-interface) QCF method is
stable in dimension greater than one.

Many blended a/c coupling methods have been proposed in the literature, e.g.,
~\cite{xiao:bridgingdomain,badia:onAtCcouplingbyblending,bridging,badia:forcebasedAtCcoupling,seleson:bridgingmethods,fish:concurrentAtCcoupling,prudhomme:modelingerrorArlequin,bauman:applicationofArlequin,XiBe:2004}.
In the present work, we formulate a blended force-based quasicontinuum
(B-QCF) method, similar to the method proposed in \cite{Lu.bqcf:2011},
which smoothly blends the forces of the atomistic and continuum model
instead of the sharp transition in the QCF method. In 1D and 2D, we
establish sharp conditions under which a linearized B-QCF operator is
positive definite.

Our results have three advantages over the stability result proven in
\cite{Lu.bqcf:2011}. Firstly, we establish $H^1$-stability (instead of
$H^2$-stability) which opens up the possibility to include defects in
the analysis, along the lines of
\cite{OrtnerShapeev:2010,DobShapOrt:2011}. Secondly, our conditions
for the positive definiteness of the linearized B-QCF operator
are needed to ensure the convergence of several popular iterative solution
methods for the B-QCF equations~\cite{qcf.iterative,luskin.iter.stat}.
We note that the convergence of these popular iterative solution
methods for the QCF equations cannot be guaranteed because of its
indefinite linearized operator~\cite{qcf.iterative,luskin.iter.stat}.  Thirdly, our results admit
much narrower blending regions, which is crucial for the computational
efficiency of the method.

The remainder of the paper is split into two sections: In
Section \ref{1DBQCFsection} we analyze positivity of the B-QCF
operator in a 1D model, whereas in Section \ref{2DBQCFsection} we
analyze a 2D model. Our methods and results are likely more widely applicable to
other force-based model couplings.

\section{Analysis of the B-QCF Operator in $1$D}\label{1DBQCFsection}
\subsection{Notation}
\label{notation}
We denote the scaled reference lattice by $\eps \mathbb{Z}:=
\{\eps\ell : \ell\in\mathbb{Z}\}$. We apply a macroscopic strain $F >
0$ to the lattice, which yields
\[
\mathbf{y}_F := F\eps \mathbb{Z} = (F \eps \ell)_{\ell \in \mathbb{Z}}.
\]
The space $\mathcal{U}$ of $2N$-periodic zero mean displacements
$\mathbf{u}=(u_{\ell})_{\ell \in \mathbb{Z}}$ from $\mathbf{y}_{F}$ is
given by
\[
\mathcal{U}:=\bigg\{\mathbf{u} : u_{\ell+2N}=u_{\ell}
\text{ for }\ell\in \mathbb{Z},
\text{ and }{\textstyle \sum_{\ell=-N+1}^{N}u_{\ell}}=0\bigg\},
\]
and we thus admit deformations $\mathbf{y}$ from the space
\[
\mathcal{Y}_{F}:=\{\mathbf{y}:
\mathbf{y}=\mathbf{y}_{F}+\mathbf{u}\text{ for some }\mathbf{u}\in
\mathcal{U}\}.
\]
We set $\eps=1/N$ throughout so that the reference length of the
computational cell remains fixed.

We define the discrete differentiation operator, $D\mathbf{u}$, on
periodic displacements by
\[
(D\mathbf{u})_{\ell}:=\frac{u_{\ell}-u_{\ell-1}}{\epsilon}, \quad
-\infty<\ell<\infty.
\]
We note that $\left(D\mathbf{u}\right)_{\ell}$ is also $2N$-periodic
in $\ell$ and satisfies the zero mean condition. We will denote
$\left(D\mathbf{u}\right)_{\ell}$ by $Du_{\ell}$.
We then define $\left(D^{(2)}\mathbf{u}\right)_{\ell}$ and $\left(D^{(3)}\mathbf{u}\right)_{\ell}$
for $-\infty<\ell<\infty$ by
\begin{equation*}
\left(D^{(2)}\mathbf{u}\right)_{\ell}:=\frac{Du_{\ell+1}-Du_{\ell}}{\epsilon};\quad
\left(D^{(3)}\mathbf{u}\right)_{\ell}:=\frac{Du^{(2)}_{\ell}-Du^{(2)}_{\ell-1}}{\epsilon}.
\end{equation*}
To make the formulas more concise we sometimes denote $Du_{\ell}$ by
$u'_{\ell}$, $D^{(2)}u_{\ell}$ by $u''_{\ell}$, etc., when there is no
confusion in the expressions.

For a displacement $\mathbf{u}\in \mathcal{U}$ and its discrete derivatives, we employ the weighted
discrete $\ell_{\epsilon}^{2}$  and $\ell_{\epsilon}^{\infty}$ norms by
\begin{align*}
\|\mathbf{u}\|_{\ell_{\epsilon}^{2}}&:= \left( \epsilon
\sum_{\ell=-N+1}^{N}|u_{\ell}|^{2}\right)^{1/2},\qquad
\|\mathbf{u}\|_{\ell_{\epsilon}^{\infty}}:=\max\limits_{-N+1\le \ell\le N}|u_{\ell}|,
\end{align*}
and the weighted inner product
\[
\la \mathbf{u},\mathbf{w}\ra :=\sum\limits_{\ell=-N+1}^{N}\epsilon u_{\ell}w_{\ell}.
\]
We will frequently use the following summation by parts identity:
\begin{lemma}[Summation by parts]
  Suppose $\{f_{k}\}_{k = m}^{n+1}$ and $\{g_{k}\}_{k = m}^{n+1}$ are two
  sequences, then
\begin{equation*}
\sum\limits_{k=m}^{n}f_{k}\left(g_{k+1}-g_{k}\right)
=\left[f_{n+1}g_{n+1}-f_{m}g_{m}\right]-\sum_{k = n}^{m}g_{k+1}\left(f_{k+1}-f_{k}\right).
\end{equation*}
\end{lemma}
Also for future reference, we state a discrete Poincar{\'e} inequality
\cite{Ortner:2008a},
\[
\|\mathbf{v}\|_{\ell_{\epsilon}^{\infty}}\le \|D\mathbf{v}\|_{\ell_{\epsilon}^1}\quad\text{for all}\, \mathbf{v}\in\mathcal{U}.
\]

\subsection{The next-nearest neighbor atomistic model and local QC approximation.}
We consider a one-dimensional ($1$D) atomistic chain with periodicity
$2N$, denoted ${\bf y} \in \mathcal{Y}$. The total atomistic energy
per period of ${\bf y}$ is given by
$\mathcal{E}^{a}(\mathbf{y})-\epsilon
\sum_{\ell=-N+1}^{N}f_{\ell}y_{\ell}$, where
\begin{equation}\label{AtomEnergy1D}
\mathcal{E}^{a}(\mathbf{y})
=\epsilon\sum_{\ell=-N+1}^{N}\left[\phi(y'_{\ell})+\phi(y'_{\ell}+y'_{\ell-1})\right]
\end{equation}
for a scaled Lennard-Jones type potential
\cite{LennardJones:1924a,Morse:1929a} $\phi$ and external forces
$f_{\ell}$.  The equilibrium equations are given by the force balance
at each atom: $F_\ell^a + f_\ell = 0$ where
\begin{align}\label{AtomEquil1D}
F_{\ell}^{a}(\mathbf{y}):=\frac{-1}{\epsilon}\frac{\partial \mathcal{E}^{a}(\mathbf{y})}{\partial y_{\ell}}
=& \frac{1}{\epsilon}\Big\{ \left[\phi'(y'_{\ell+1})+\phi'(y'_{\ell+2}+y'_{\ell+1})\right]
-\left[\phi'(y'_{\ell})+\phi'(y'_{\ell}+y'_{\ell-1})\right]
\Big\}.
\end{align}
We assume that the displacement $\mathbf{u}^a = \mathbf{y}^a -
\mathbf{y}_F$ is ``small'' and hence linearize the atomistic
equilibrium equations about $\mathbf{y}_{F}$ to obtain
\[
\left(L^{a}\mathbf{u}^a\right)_{\ell}=f_{\ell},\quad\text{for}\quad\ell=-N+1,\dots,N,
\]
where $\left(L^a\mathbf{v}\right)$ for a displacement $\mathbf{v}\in \mathcal{U}$ is given by
\[
\left(L^a\mathbf{v}\right)_{\ell}:=\phi''_{F}\frac{\left(-v_{\ell+1}+2v_{\ell}-v_{\ell-1}\right)}{\epsilon^{2}}+
\phi''_{2F}\frac{\left(-v_{\ell+2}+2v_{\ell}-v_{\ell-2}\right)}{\epsilon^{2}}.
\]
Here and throughout we use the notation $\phi''_{F}:=\phi''(F)$ and
$\phi''_{2F}:=\phi''(2F)$, where $\phi$ is the potential in
\eqref{AtomEnergy1D}.  We assume that $\phi''_{F} > 0$, which holds
for typical pair potentials such as the Lennard-Jones potential under
physically relevant deformations.

We will later require the following characterisation of the stability
of $L^a$.

\begin{lemma}
  \label{th:stab_atm}
  $L^a$ is positive definite, uniformly for $N \in \mathbb{N}$, if and
  only if $c_0 := \min(\phi_F'', \phi_F'' + 4 \phi_{2F}'') >
  0$. Moreover,
  \begin{displaymath}
    \la L^a \mathbf{u},\mathbf{u}\ra \geq c_0
    \|D\mathbf{u}\|^2_{\ell_{\epsilon}^2} \qquad \forall \mathbf{u}
    \in \mathcal{U}.
  \end{displaymath}
\end{lemma}
\begin{proof}
  The case $\phi_{2F}'' \leq 0$ was treated in
  \cite{doblusort:qce.stab}, hence suppose that $\phi_{2F}'' > 0$. The
  coercivity estimate is trivial in this case, and it remains to show
  that it is also sharp. To that end, we note that
  \begin{displaymath}
    \la L^a \mathbf{u},\mathbf{u}\ra = \eps \sum_{\ell} \phi_F''
    (u_\ell')^2 + \eps \sum_{\ell} \phi_{2F}'' (u_{\ell-1}' + u_\ell')^2.
  \end{displaymath}
  Hence, testing with $u_\ell' = (-1)^\ell$ (this is admissible since
  there is an even number of atoms per period), the second-neighbor
  terms drop out and we obtain $\la L^a \mathbf{u},\mathbf{u}\ra =
  \phi_F''   \|D\mathbf{u}\|^2_{\ell_{\epsilon}^2}$.
\end{proof}

The local QC approximation (QCL) uses the Cauchy--Born extrapolation rule \cite{Ortiz:1995a,Shimokawa:2004},
that is, approximating $y'_{\ell}+y'_{\ell-1}$  in \eqref{AtomEnergy1D} by $2y'_{\ell}$
in our context. Thus, the QCL energy is given by
\begin{equation}\label{QCLEnergy1D}
\mathcal{E}^{qcl}(\mathbf{y})=\epsilon\sum_{\ell=-N+1}^{N}\left[\phi(y'_{\ell})+\phi(2y'_{\ell})\right].
\end{equation}
We can similarly obtain the linearized QCL equilibrium equations about the uniform deformation
\[
\left(L^{qcl}\mathbf{u}^{qcl}\right)_{\ell}=f_{\ell}\quad \text{for}\quad \ell=-N+1,\dots, N,
\]
where the expression of $\left(L^{qcl}\mathbf{v}\right)_{\ell}$ with $\mathbf{v}\in \mathcal{U}$
is
\[
\left(L^{qcl}\mathbf{v}\right)_{\ell}:=
\left(\phi''_{F}+4\phi''_{2F}\right)\frac{\left(-v_{\ell+1}+2v_{\ell}-v_{\ell-1}\right)}{\epsilon^2}.
\]

\subsection{The Blended QCF Operator}
The blended QCF (B-QCF) operator is obtained through smooth blending
of the atomistic and local QC models. Let $\beta : \mathbb{R} \to
\mathbb{R}$ be a ``smooth'' and $2$-periodic blending function, then
we define
\begin{displaymath}
  F_\ell^{bqcf}(\mathbf{y}) := \beta_\ell F_\ell^a(\mathbf{y}) +
  (1-\beta_\ell) F_\ell^{qcl}(\mathbf{y}),
\end{displaymath}
where $F_\ell^{qcl}$ is defined analogously to $F_\ell^a$ and
$\beta_\ell := \beta(F\epsilon \ell)$. Linearisation about
$\mathbf{y}_F$ yields the linearized B-QCF operator
\[
(L^{bqcf}\mathbf{v})_{\ell}:=\beta_{\ell} (L^a\mathbf{v})_{\ell}+(1-\beta_{\ell})(L^{qcl}\mathbf{v})_{\ell}.
\]
In order to obtain a {\em practical} atomistic-to-continuum coupling
scheme, we would also need to coarsen the continuum region by choosing
a coarser finite element mesh. In the present work we focus
exclusively on the stability of the B-QCF operator, which is a
necessary ingredient in any subsequent analysis of the B-QCF method.

\subsection{Positive-Definiteness of the B-QCF Operator}
We begin by writing $L^{bqcf}$ in the form
$L^{bqcf}=\phi''_{F}L_{1}^{bqcf}+\phi''_{2F}L_{2}^{bqcf}$ where
\begin{align*}
\left(L_{1}^{bqcf}\mathbf{v}\right)_{\ell}=&\epsilon^{-2}\left(-v_{\ell+1}+2v_{\ell}-v_{\ell-1}\right),\quad \text{and}\\
\left(L_{2}^{bqcf}\mathbf{v}\right)_{\ell}=&\beta_{\ell}\epsilon^{-2}\left(-v_{\ell+2}+2v_{\ell}-v_{\ell-2}\right)
+(1-\beta_{\ell})4\epsilon^{-2}\left(-v_{\ell+1}+2v_{\ell}-v_{\ell-1}\right).
\end{align*}
\begin{lemma}\label{DivformLemma}
For any $\mathbf{u}\in \mathcal{U}$, the nearest neighbor and next-nearest neighbor interaction operator can be written in the
form
\begin{align}\label{Divform}
\begin{split}
\la L^{bqcf}_{1}\mathbf{u},\mathbf{u}\ra=&\|D\mathbf{u}\|_{\ell_{\epsilon}^2}^2,\quad \text{and}\\
\la
L^{bqcf}_{2}\mathbf{u},\mathbf{u}\ra=& \big[ 4\|D\mathbf{u}\|_{\ell_{\epsilon}^2}^2
-\epsilon^{2}\|\sqrt{\beta}D^{(2)}\mathbf{u}\|_{\ell_{\epsilon}^2}^2 \big]
+\mathbf{R}+\mathbf{S}+\mathbf{T},
\end{split}
\end{align}
where the terms $\mathbf{R}$ and $\mathbf{S}$ are given by
\begin{align}\label{SigPart}
\begin{split}
\mathbf{R}=&\sum\limits_{\ell=-N+1}^{N}2\epsilon^3D^{(2)}\beta_{\ell}\left(Du_{\ell}\right)^2 ,
\quad
\mathbf{S}=\sum\limits_{\ell=-N+1}^{N}\epsilon^4 D^{(2)}\beta_{\ell}D^{(2)}u_{\ell}Du_{\ell}\quad \\
&\qquad\qquad\text{and}\quad
\mathbf{T}=\sum\limits_{\ell=-N+1}^{N}\epsilon^3 \left(D^{(3)}\beta_{\ell+1}\right)u_{\ell}Du_{\ell+1}.
\end{split}
\end{align}
\end{lemma}
\begin{proof} Since the proof of the first identity of
Lemma~\ref{DivformLemma} is not difficult, we only prove the
identity for $L^{bqcf}_{2}$. The main tool used here is the
summation by parts formula. We note that
\begin{align}
\la L^{bqcf}_{2}\mathbf{u},\mathbf{u}\ra
=&\sum\limits_{\ell=-N+1}^{N} \epsilon \beta_{\ell} \frac{\left(-u_{\ell+2}+2u_{\ell}-u_{\ell-2}\right)}{\epsilon^2}u_{\ell}
+\epsilon(1-\beta_{\ell})\frac{4\left( -u_{\ell+1}+2u_{\ell}-u_{\ell-1}\right)}{\epsilon^2}u_{\ell}\nonumber\\
=&\sum\limits_{\ell=-N+1}^{N}\epsilon\frac{4\left(-u_{\ell+1}+2u_{\ell}-u_{\ell-1}\right)}{\epsilon^2}u_{\ell}\nonumber\\
&\qquad\qquad+\sum\limits_{\ell=-N+1}^{N}\epsilon \beta_{\ell}\frac{\left(-u_{\ell+2}+4u_{\ell+1}-6u_{\ell}+4u_{\ell-1}-u_{\ell-2}\right)}{\epsilon^2}
u_{\ell}\nonumber\\
=&4\|D\mathbf{u}\|^2_{\ell_{\epsilon}^2}
+\sum\limits_{\ell=-N+1}^{N}\epsilon^2\beta_{\ell} \left(-D^{(3)}u_{\ell+1}+D^{(3)}u_{\ell}\right)u_{\ell}.\label{Divlemeq1}
\end{align}
We then apply the summation by parts formula to the second term of \eqref{Divlemeq1} to obtain
\[
\begin{split}
&\sum\limits_{\ell=-N+1}^{N}\beta_{\ell}\epsilon^2 \left(-D^{(3)}u_{\ell+1}+D^{(3)}u_{\ell}\right)u_{\ell}\\
&\qquad=\sum\limits_{\ell=-N+1}^{N}\epsilon^2 D^{(3)}u_{\ell+1}\left[\beta_{\ell+1}u_{\ell+1}-\beta_{\ell}u_{\ell}\right]
=\sum\limits_{\ell=-N+1}^{N}\epsilon^3 D^{(3)}u_{\ell}\left[\beta_{\ell}Du_{\ell}+u_{\ell-1}D\beta_{\ell}\right].
\end{split}
\]
We use the summation by parts formula again and change the index
according to the periodicity so that we get
\begin{align}
\sum\limits_{\ell=-N+1}^{N}&\epsilon^3
D^{(3)}u_{\ell}\left[\beta_{\ell}Du_{\ell}+u_{\ell-1}D\beta_{\ell}\right]\nonumber\\
&= \sum_{\ell=-N+1}^{N}\epsilon^2
\left(\beta_{\ell}Du_{\ell}\right)\left(D^{(2)}u_{\ell}-D^{(2)}u_{\ell-1}\right)
+\sum\limits_{\ell=-N+1}^{N}\epsilon^3
\left(D^{(3)}u_{\ell}\right)\,u_{\ell-1}D\beta_{\ell}\nonumber\\
&= \sum\limits_{\ell=-N+1}^{N}\epsilon^2
\left(-D^{(2)}u_{\ell}\right)\left(\beta_{\ell+1}Du_{\ell+1}-\beta_{\ell}Du_{\ell}\right)
+\sum\limits_{\ell=-N+1}^{N}\epsilon^3
\left(D^{(3)}u_{\ell}\right)\,u_{\ell-1}D\beta_{\ell}\nonumber\\
&= \sum\limits_{\ell=-N+1}^{N}\epsilon^2
\left(-D^{(2)}u_{\ell}\right)\left[\beta_{\ell+1}Du_{\ell+1}-\beta_{\ell}Du_{\ell+1}+\beta_{\ell}Du_{\ell+1}-\beta_{\ell}Du_{\ell}\right]
\nonumber\\&\qquad\qquad +\sum\limits_{\ell=-N+1}^{N}\epsilon^3
\left(D^{(3)}u_{\ell}\right)\,u_{\ell-1}D\beta_{\ell}\nonumber\\
 &= -\epsilon^2
\|\sqrt{\beta}D^{(2)}\mathbf{u}\|^2_{\ell_{\epsilon}^2} +
\sum\limits_{\ell=-N+1}^{N}\epsilon^3\left[-D^{(2)}u_{\ell-1}D\beta_{\ell}Du_{\ell}+
D^{(3)}u_{\ell}\,u_{\ell-1}D\beta_{\ell}\right].\label{Divlemeq2}
\end{align}
We now focus on the second term of \eqref{Divlemeq2}. We
repeatedly use the summation by parts formula to obtain
\begin{align*}
\sum\limits_{\ell=-N+1}^{N}&\epsilon^3\left[-D^{(2)}u_{\ell-1}D\beta_{\ell}Du_{\ell}+
\left(D^{(3)}u_{\ell}\right)\,u_{\ell-1}D\beta_{\ell}\right]\\
&=
\sum\limits_{\ell=-N+1}^{N}-\epsilon^2D\beta_{\ell}\left[\left(Du_{\ell}\right)^2-\left(Du_{\ell-1}\right)^2\right]\\
&\qquad+\sum\limits_{\ell=-N+1}^{N}\epsilon^2D\beta_{\ell}\left[\left(Du_{\ell}-Du_{\ell-1}\right)Du_{\ell-1}
+\left(D^{(2)}u_{\ell}-D^{(2)}u_{\ell-1}\right)u_{\ell-1}\right]\\
&= \sum\limits_{\ell=-N+1}^{N}\epsilon^3D^{(2)}\beta_{\ell}\left(Du_{\ell}\right)^2
+\sum\limits_{\ell=-N+1}^{N}\epsilon^2D\beta_{\ell}\left[u_{\ell-1}D^{(2)}u_{\ell}-u_{\ell-2}D^{(2)}u_{\ell-1}\right]\\
&=\sum\limits_{\ell=-N+1}^{N}2\epsilon^3 D^{(2)}\beta_{\ell}\left(Du_{\ell}\right)^2
+\sum\limits_{\ell=-N+1}^{N}\epsilon^4 D^{(2)}\beta_{\ell}D^{(2)}u_{\ell}Du_{\ell}
+\sum\limits_{\ell=-N+1}^{N}\epsilon^3 \left(D^{(3)}\beta_{\ell+1}\right)u_{\ell}Du_{\ell+1}\\
&=\mathbf{R}+\mathbf{S}+\mathbf{T},
\end{align*}
where $\mathbf{R}$, $\mathbf{S}$ and $\mathbf{T}$ are defined in \eqref{SigPart}.

Combining all of the above equalities, we obtain
\eqref{Divform}.
\end{proof}

We shall see below that the first group in \eqref{Divform} does not
negatively affect the stability of the B-QCF operator. By contrast,
the three terms ${\bf R}$, ${\bf S}$, ${\bf T}$ should be considered
``error terms''.  We estimate them in the next lemma.

In order to proceed with the analysis we define
\begin{displaymath}
  \mathcal{I}:=\big\{\ell \in \mathbb{Z} : 0 < \beta_{\ell+j} < 1
  \text{ for some } j \in \{\pm 1, \pm 2\} \big\},
\end{displaymath}
so that $D^{(j)}\beta_\ell = 0$ for all $\ell \in \{-N+1,\dots N\}
\setminus \mathcal{I}$ and $j \in \{1,2,3\}$, and
$K:=\sharp\mathcal{I}$.

\begin{lemma}
  \label{SigEstLemma}
  Let $\mathbf{R}$, $\mathbf{S}$ and $\mathbf{T}$ be defined by
  \eqref{SigPart}, then we have the following estimates:
  \begin{equation}\label{SigEst}
    \begin{split}
      |\mathbf{R}|\le~&
      \epsilon^{2}\|D^{(2)}\beta\|_{\ell_{\epsilon}^{\infty}}\|D\mathbf{u}\|^2_{\ell_{\epsilon}^2},
      \\
      |{\bf S}| \leq~& 2\epsilon^2
    \|D^{(2)}\beta\|_{\ell_{\epsilon}^{\infty}}\|D\mathbf{u}\|^2_{\ell_{\epsilon}^2},
    \quad \text{and} \\
      |\mathbf{T}| \le~&
      \epsilon^2 \sqrt{2}(K\epsilon)^{1/2} \|D^{(3)}\beta\|_{\ell_{\epsilon}^{\infty}} \,
      \|D\mathbf{u}\|^2_{\ell_{\epsilon}^2}.
    \end{split}
  \end{equation}
\end{lemma}
\begin{proof}
  The estimate for ${\bf R}$ follows directly from H{\"o}lder's
  inequality.

  To estimate ${\bf S}$ recall that
  $D^{(2)}u_{\ell}:=\frac{Du_{\ell+1}-Du_{\ell}}{\epsilon}$, which
  implies
  \[
  \|D^{(2)}\mathbf{u}\|^2_{\ell_{\epsilon}^2}\le
  \frac{4}{\epsilon^2}\|D\mathbf{u}\|^2_{\ell_{\epsilon}^2}.
  \]
  Therefore, ${\bf S}$ is bounded by
  \begin{align*}
    |{\bf S}| = \left| \sum\limits_{\ell=-N+1}^{N} \epsilon^4
      D^{(2)}\beta_{\ell}D^{(2)}u_{\ell}Du_{\ell}\right| \le\epsilon^3
    \|D^{(2)}\beta\|_{\ell_{\epsilon}^{\infty}}\|D^{(2)}\mathbf{u}\|_{\ell_{\epsilon}^2}
    \|D\mathbf{u}\|_{\ell_{\epsilon}^2}\le 2\epsilon^2
    \|D^{(2)}\beta\|_{\ell_{\epsilon}^{\infty}}\|D\mathbf{u}\|^2_{\ell_{\epsilon}^2}.
  \end{align*}

  Finally, we estimate ${\bf T}$ by
\[
|{\bf T}| = \left|\sum\limits_{\ell=-N+1}^{N}\epsilon^3D^{(3)}\beta_{\ell+1}Du_{\ell+1}\,u_{\ell}\right|
\le \epsilon^2
\|D^{(3)}\beta\|_{\ell_{\epsilon}^{\infty}}\|\mathbf{u}\|_{\ell_{\epsilon}^2(\mathcal{I})}\|D\mathbf{u}\|_{\ell_{\epsilon}^2},
\]
We then apply the H{\"o}lder inequality, the Poincar{\'e} inequality
and Jensen's inequality successively to
$\|\mathbf{u}\|_{\ell_{\epsilon}^2(\mathcal{I})}$ to get
\begin{equation*}
\|\mathbf{u}\|^2_{\ell_{\epsilon}^2(\mathcal{I})} \le
(K\epsilon)\|\mathbf{u}\|^2_{\ell_{\epsilon}^{\infty}}
\le K\epsilon \|D\mathbf{u}\|^2_{\ell_{\epsilon}^{1}} \le 2K\epsilon
\|D\mathbf{u}\|^2_{\ell_{\epsilon}^{2}}.
\end{equation*}
Therefore, we have
\[
\left|{\bf T}\right|
\le \epsilon^2
\|D^{(3)}\beta\|_{\ell_{\epsilon}^{\infty}}\|\mathbf{u}\|_{\ell_{\epsilon}^2(\mathcal{I})}\|D\mathbf{u}\|_{\ell_{\epsilon}^2}
\le\sqrt{2} \epsilon^2
\|D^{(3)}\beta\|_{\ell_{\epsilon}^{\infty}}\left(K\epsilon\right)^{1/2}\|D\mathbf{u}\|^2_{\ell_{\epsilon}^2}.
\]
Combining the above estimates, we have proven the second
inequality in \eqref{SigEst}.
\end{proof}

We see from the previous result that smoothness of $\beta$ crucially
enters the estimates on the error terms ${\bf R}$, ${\bf S}$, ${\bf
  T}$. Before we state our main result in 1D we show how quasi-optimal
blending functions can be constructed to minimize these terms, which
will require us to introduce the {\em blending width} into the
analysis. The estimate \eqref{eq:BlendFunEst_upper} is stated for a
single connected interface region, however, an analogous result holds
if the interface has connected components with comparable width. A
similar result can also be found in \cite{bqce11}.

\begin{lemma}\label{BlendFunEstLemma}
  \begin{itemize}
  \item[(i)] Suppose that the blending region is connected, that is
    $\mathcal{I} = \{ 1, \dots, K\}$ without loss of generality, then
    $\beta$ can be chosen such that
    \begin{equation}
      \label{eq:BlendFunEst_upper}
      \| D^{(j)} \beta \|_{\ell^\infty} \leq C_\beta (K \epsilon)^{-j},
      \quad \text{for } j = 1, 2, 3,
    \end{equation}
    where $C_\beta$ is independent of $K$ and $\epsilon$.

  \item[(ii)] This estimate is sharp in sense that, if $\beta_\ell$
    attains both the values $0$ and $1$, then
    \begin{equation}\label{BlendFunEst}
      \|D^{(j)}\beta\|_{\ell^{\infty}}\ge (K\epsilon)^{-j},\quad
      \text{for } j=1,2,3.
    \end{equation}

  \item[(iii)] Suppose that $\mathcal{J} = \{1, \dots, n\} \subset
    \mathcal{I}$ such that $\beta(1) = 0$, $\beta(n) = 1$ (or
    vice-versa), and $0, n+1 \notin \mathcal{I}$, and suppose moreover
    that \eqref{eq:BlendFunEst_upper} is satisfied, then
    \begin{equation}
      \label{eq:blend_interval}
      \# \big\{ \ell \in \mathcal{J} : D^{(3)}\beta_\ell \leq
      -\smfrac12 (\epsilon K)^{-3} \big\}
      \geq \smfrac{1}{2C_\beta} K.
    \end{equation}
 \end{itemize}
\end{lemma}
\begin{proof}
  {\it (i) } The upper bound follows by fixing a reference blending
  function $B \in C^3(\mathbb{R})$, $B = 0$ in $(-\infty, 0]$ and $B =
  1$ in $[1, +\infty)$, and defining $\beta(x) = B((x-2\epsilon) /
  (\epsilon K'))$ for $K' = K-4$. Then $\mathcal{I} = \{1, \dots,
  K\}$, and a scaling argument immediately gives
  \eqref{eq:BlendFunEst_upper}.

  {\it (ii) } To prove the lower bound, suppose $0 < \beta_\ell < 1$
  for $\ell = 1, \dots, K_0-1$, and $\beta_0 = 0$ and $\beta_{K_0} =
  1$. Then $\epsilon \sum_{\ell = 1}^{K_0} \beta_\ell' = 1$, from
  which infer the existence of $K_1 \in \{1, \dots, K_0\}$ such that
  $\beta_{K_1}' \geq 1 / (\epsilon K_0)$. This establishes the lower
  bound for $j = 1$. To prove it for $j = 2$ we note that, since
  $\beta_{K_0} = 1$, $\beta_{K_0+1}' \leq 0$, and hence we obtain
  \begin{displaymath}
    \epsilon \sum_{\ell = K_1+1}^{K_0} \beta_\ell'' = \beta_{K_0+1}' -
    \beta_{K_1}' \leq - 1 / (\epsilon K_0).
  \end{displaymath}
  We deduce that there exists $K_2$
  such that $\beta_{K_2}'' \leq - 1 / (\epsilon^2 K_0 (K_0-K_1))\le
  -1/(\epsilon K)^2$.  This implies \eqref{BlendFunEst} for $j =
  2$. We can argue similarly to obtain the result for $j = 3$.

  {\it (iii) } Finally, to establish \eqref{eq:blend_interval}, let $m
  \in \mathbb{N}$ be chosen minimally such that $\beta_m'' \leq -
  (\epsilon K)^{-2}$ and $\beta_0'' = 0$; then $m \leq n$ and we have
  \begin{displaymath}
    - \frac{1}{(\epsilon K)^2} \geq \beta_m'' - \beta_0'' = \epsilon
    \sum_{\ell = 1}^{m} \beta_\ell''' \geq - \frac{\epsilon k
      C_\beta}{(\epsilon K)^3} - \frac{\epsilon(m - k)}{2 (\epsilon K)^3},
  \end{displaymath}
  where $k := \#\{ \ell \in \mathcal{J} : \beta_\ell''' \leq - \frac12
  (\epsilon K)^{-3} \}$. Rearranging the inequality, we obtain
  \begin{displaymath}
    - \frac{1}{2 (\epsilon K)^2} \geq
    - \frac{1}{(\epsilon K)^2} +  \frac{\epsilon(m - k)}{2 (\epsilon
      K)^3}
    \geq - \frac{\epsilon k
      C_\beta}{(\epsilon K)^3} \geq - \frac{k C_\beta}{K (\epsilon K)^2},
  \end{displaymath}
  and we immediately deduce that $k / K \geq 1 / (2C_\beta)$, which
  concludes the proof of item (iii).
\end{proof}

We can summarize the previous estimates and get the following optimal
condition for the size $K$ of the blending region provided that
$\beta$ is chosen in a quasi-optimal way. Formally, the result states
that $L^{bqcf}$ is positive definite if and only if $K \gg
\epsilon^{-1/5}$. In particular, we conclude that the B-QCF operator
is positive definite for fairly moderate blending widths.

\begin{theorem}\label{BlendSizeThm}
  Let $\mathcal{I}$ and $K$ be defined as in
  Lemma~\ref{BlendFunEstLemma}, and suppose that $\beta$ is chosen to
  satisfy the upper bound \eqref{eq:BlendFunEst_upper}. Then there
  exists a constant $C_1 = C_1(C_\beta)$, 
  such that
  \begin{equation}
    \label{eq:1d_coerc_lower}
    \langle L^{bqcf} {\bf u}, {\bf u} \rangle \geq \big(c_0 - C_1
    |\phi_{2F}''| \big[ K^{-5/2} \epsilon^{-1/2}\big]\big) \|D {\bf
      u}\|_{\ell^2_\epsilon}^2
    \qquad \forall {\bf u} \in \mathcal{U},
  \end{equation}
  where $c_0 = \min(\phi_F'', \phi_F'' + 4 \phi_{2F}'')$ is the
  atomistic stability constant 
of Lemma \ref{th:stab_atm}.

  Moreover, if $\beta_\ell$ takes both the values $0$ and $1$, then
  there exist constants $C_2, C_3 > 0$, independent of $\mathcal{I}$,
  $N$, $\phi_F''$ and $\phi_{2F}''$, such that
  \begin{equation}
    \label{eq:1d_coerc_upper}
    \inf_{\substack{{\bf u} \in \mathcal{U} \\ \|D{\bf
          u}\|_{\ell^2_\epsilon} = 1}} \langle L^{bqcf} {\bf u}, {\bf
      u} \rangle \leq \phi_F'' + C_2 |\phi_{2F}''| - C_3 |\phi_{2F}''|
    \big[ K^{-5/2} \epsilon^{-1/2} \big].
  \end{equation}
\end{theorem}

\begin{remark}
  Estimates \eqref{eq:1d_coerc_lower} and \eqref{eq:1d_coerc_upper}
  establish the asymptotic optimality of the blending width $K \eqsim
  \epsilon^{-1/5}$ in the limit as $\epsilon \to 0$:
  \eqref{eq:1d_coerc_lower} implies that, if $c_0 > 0$ and $K \gg
  \epsilon^{-1/5}$, then $L^{bqcf}$ is coercive, while
  \eqref{eq:1d_coerc_upper} shows that, if $K \ll \epsilon^{-1/5}$
  then $L^{bqcf}$ is necessarily indefinite. \hfill \qed
\end{remark}

\begin{proof}
  We first prove the lower bound. The blended force-based operator
  satisfies $L^{bqcf}$
  \[
  \la L^{bqcf}\mathbf{u},\mathbf{u}\ra
  =A_F\|D\mathbf{u}\|^2_{\ell_{\epsilon}^2}-\epsilon^2
  \phi''_{2F}\|\sqrt{\beta}D^{(2)}\mathbf{u}\|^2_{\ell_{\epsilon}^2}
  +\phi''_{2F}(\mathbf{R}+\mathbf{S}+\mathbf{T})
  \]
  where $A_F:=\phi''_{F}+4\phi''_{2F}.$ From Lemma~\ref{SigEstLemma},
  we have \begin{equation*}
    |\mathbf{R}+\mathbf{S}+\mathbf{T}|
\le \epsilon^2\left[
4\|D^{(2)}\beta\|_{\ell_{\epsilon}^{\infty}}
+(K\epsilon)^{1/2}\|D^{(3)}\beta\|_{\ell_{\epsilon}^{\infty}}\right]\|D\mathbf{u}\|^2_{\ell_{\epsilon}^2}.
\end{equation*}
Since
$\|D^{(j)}\beta\|_{\ell_{\epsilon}^{\infty}}\le
C_{\beta}(K\epsilon)^{-j}$, so we have
\begin{align*}
|\mathbf{R}+\mathbf{S}+\mathbf{T}| \le C \epsilon^2\left[ 4(K\epsilon)^{-2}
+(K\epsilon)^{1/2}(K\epsilon)^{-3}\right]\|D\mathbf{u}\|^2_{\ell_{\epsilon}^2}
\le C_3 \left[K^{-5/2}\epsilon^{-1/2}\right]\|D\mathbf{u}\|^2_{\ell_{\epsilon}^2},
\end{align*}
where we used the fact that $K^{-2} \leq K^{-5/2} \epsilon^{-1/2}$.

If $\phi''_{2F}\le 0$, then we obtain
\[
\la L^{bqcf}\mathbf{u},\mathbf{u}\ra \ge \left(A_F-C_1|\phi''_{2F}|
\left[ K^{-5/2}\epsilon^{-1/2}\right]\right)\|D\mathbf{u}\|^2_{\ell_{\epsilon}^2}.
\]
If $\phi''_{2F} > 0$, then
\begin{align*}
  \la L^{bqcf}_{2}\mathbf{u},\mathbf{u}\ra
  =~& A_F\|D\mathbf{u}\|^2_{\ell_{\epsilon}^2}-\epsilon^2
\phi''_{2F}\|\sqrt{\beta}D^{(2)}\mathbf{u}\|^2_{\ell_{\epsilon}^2}+\phi''_{2F}(\mathbf{R}+\mathbf{S}+\mathbf{T})
\\
\geq~& \left(\phi''_{F}-C_3|\phi''_{2F}|
  \left[K^{-5/2}\epsilon^{-1/2}\right]\right)\|D\mathbf{u}\|^2_{\ell_{\epsilon}^2},
\end{align*}
which is the corresponding result.

To prove the opposite bound, let $\mathcal{J}$ be defined as in Lemma
\ref{BlendFunEstLemma} (iii). We can assume this without loss of
generality upon possibly shifting and inverting the blending function.
We define $\mathcal{J}' := \{ \ell \in \mathcal{J} : D^{(3)}
\beta_\ell \leq - \smfrac12 (\epsilon K)^{-3} \}$ and $L := \epsilon
\#\mathcal{J}' = \alpha \epsilon K$ for some $\alpha \geq 1 /
(2C_\beta)$, and a test function $\mathbf{v} \in \mathcal{U}$ through
$v_0 = \smfrac12$ and
\begin{equation}\label{DvDef1}
\begin{split}
v_{\ell}'=\begin{cases} L^{-1/2}, &\quad \ell\in
\mathcal{J}'\\
0, & \quad \ell\in \mathcal{I}\setminus \mathcal{J}',
\end{cases}
\end{split}
\end{equation}
and extending $v_\ell'$ outside of $\mathcal{I}$ in such a way that
$\|D{\bf v}\|_{\ell^2_\epsilon}$ is bounded uniformly in $\mathcal{I}$
and $N$, and such that ${\bf v}$ is $2N$-periodic (see
\cite{qcf.stab} for details of this construction).

With these definitions we obtain
\begin{align*}
  {\bf T} =~&
  \epsilon^3\sum\limits_{\ell=-N+1}^{N}D^{(3)}\beta_{\ell+1}Dv_{\ell+1}\,v_{\ell}
  = \epsilon^3 \sum_{\ell \in \mathcal{J}'} D^{(3)}\beta_{\ell-1}
  v_{\ell}' v_{\ell-1} \\
  \leq~& - \frac{\epsilon^2 L L^{-1/2}}{4 (\epsilon K)^3} = -
  \frac{(\alpha \epsilon K)^{1/2}}{4 \epsilon K^3} = - \smfrac{\alpha^{1/2}}{4}
  K^{-5/2} \epsilon^{-1/2}.
\end{align*}
Recall that, by contrast, we have
\begin{displaymath}
  |{\bf R} + {\bf S}| \leq C_2 K^{-2} \| D{\bf v} \|_{\ell^2_\epsilon}^2.
\end{displaymath}
Combining these estimates, and using the fact that $\|D{\bf
  v}\|_{\ell^2_\epsilon}$ is bounded independently of $\mathcal{I}$
and $N$, yields \eqref{eq:1d_coerc_upper}.
\end{proof}

\section{Positive-Definiteness of the B-QCF Operator in
  $2$D}\label{2DBQCFsection}
\subsection{The triangular lattice}
For some integer $N\in\mathbb{N}$ and $\epsilon:=1/N$, we define the
scaled 2D triangular lattice
\[
\mathbb{L}:=\mathtt{A}_{6}\mathbb{Z}^2,\quad\text{where}\quad
\mathtt{A}_{6}:=\left[{a}_{1},{a}_{2}\right]
:= \epsilon\left[\begin{array}{cc}
1 & 1/2\\
0 & \sqrt{3}/2
\end{array}\right],
\]
where ${a}_{i},\,i=1,2$ are the scaled lattice vectors.
Throughout our analysis, we use the following definition of the periodic reference cell
\[
\Omega:=\mathtt{A}_{6}(-N, N]^2\quad \text{and}\quad \mathcal{L}:=\mathbb{L}\cap\Omega.
\]
We furthermore set ${a}_3=(-1/2\epsilon,
\sqrt{3}/2\epsilon)^{\mathtt{T}}$, $a_4 := - a_1, a_5 := -a_2$ and $a_6
:= -a_3$; then the set of \emph{nearest-neighbor directions} is given
by
\[
\mathcal{N}_{1}:=\{\pm{a}_1,\pm{a}_{2},\pm{a}_3\}.
\]
The set of \emph{next nearest-neighbor directions} is given by
\[
\mathcal{N}_{2}:=\{\pm {b}_1,\pm {b}_{2},\pm {b}_3\},
\quad \text{where} \quad b_1:=a_1+a_2,\quad b_2:=a_2+a_3\quad\text{and}\quad b_3=a_3-a_1.
\]
We use the notation $\mathcal{N}:=\mathcal{N}_1\cup\mathcal{N}_2$ to
denote the directions of the neighboring bonds in the interaction
range of each atom (see Figure~\ref{AtomDomainFig}).

We identify all lattice functions $\mathbf{v} : \mathbb{L} \to
\mathbb{R}^2$ with their continuous, piece affine interpolants with
respect to the canonical triangulation $\mathcal{T}$ of $\mathbb{R}^2$
with nodes $\mathbb{L}$.

\begin{figure}
\begin{center}
\subfigure[Neighbor Set]{\includegraphics[height = 5 cm ]{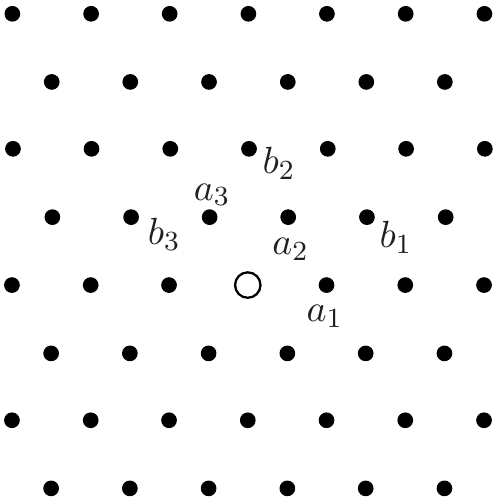}}\quad
\subfigure[Domain Decomposition]{\includegraphics [height =5 cm]{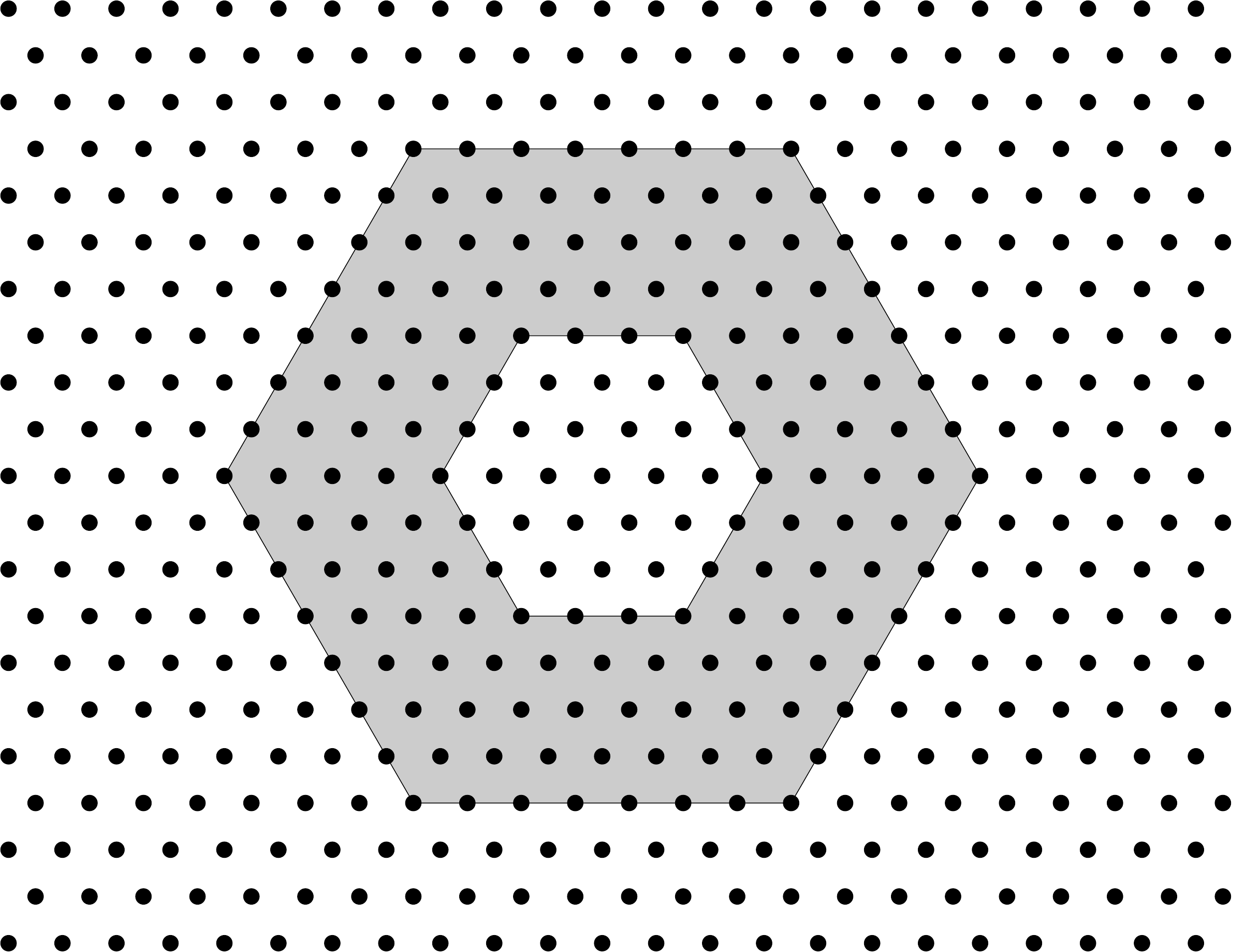}}
\end{center}
\caption{(a) The 12 neighboring bonds of each atom. (b) The atomistic
  region is $\Omega_a=\mathtt{Hex}(\epsilon R_a)$. The blending region
  is $\Omega_b=\mathtt{Hex}(\epsilon R_b)\setminus \Omega_a$. Here,
  $R_a = 3$, $R_b = 7$ and $K = 4$.}\label{AtomDomainFig}
\end{figure}
\subsection{The atomistic, continuum and blending regions}
Let $\mathtt{Hex}(r)$ denote the closed hexagon centered at the
origin, with sides aligned with the lattice directions $a_1, a_2, a_3$,
and diameter $2r$.

For $R_a< R_b \in \mathbb{N}$, we define the atomistic, blending and
continuum regions, respectively, as
\begin{displaymath}
  \Omega_a := \mathtt{Hex}(\epsilon R_a), \quad
  \Omega_b := \mathtt{Hex}(\epsilon R_b) \setminus \Omega_a,
  \quad \text{and} \quad
  \Omega_c:={\rm clos}\left(\Omega\setminus\left(\Omega_a\cup\Omega_b\right)\right).
\end{displaymath}
We denote the blending width by $K := R_b - R_a$.  Moreover, we define
the corresponding lattice sites
\begin{displaymath}
  \mathcal{L}^a := \mathcal{L}\cap \Omega_a, \qquad
  \mathcal{L}^b := \mathcal{L}\cap \Omega_b, \qquad \text{and} \qquad
  \mathcal{L}^c := \mathcal{L}\cap \Omega_c.
\end{displaymath}
For simplicity, we will again use $\mathcal{L}$ as the finite element
nodes, that is, every atom is a repatom.


For a map $\mathbf{v}:\mathbb{L}\rightarrow \mathbb{R}^2$ and bond
directions $r,s\in \mathcal{N}$, we define the finite difference
operators
\begin{displaymath}
D_{r}v(x):=\frac{v(x+r)-v(x)}{\epsilon}\quad\text{and}\quad
D_{r}D_{s}v(x):=\frac{D_{s}v(x+r)-D_{s}v(x)}{\epsilon}.
\end{displaymath}

We define the space of all admissible displacements, $\mathcal{U}$, as
all discrete functions $\mathbb{L}\rightarrow \mathbb{R}^2$ which are
$\Omega$-periodic and satisfy the mean zero condition on the computational domain:
\[
\mathcal{U}:=\Big\{\mathbf{u}:\mathbb{L}\rightarrow
\mathbb{R}^2 : \text{$u(x)$ is $\Omega$-periodic and
}{\textstyle \sum_{x\in\mathcal{L}}} u(x)=0 \Big\}.
\]
For a given matrix $B \in \mathbb{R}^{2 \times 2}$,
$\mathrm{det}(B)>0$, we admit deformations $\mathbf{y}$ from the space
\[
\mathcal{Y}_{B}:=\big\{ \mathbf{y}
:\mathbb{L}\rightarrow
\mathbb{R}^2: y(x)=Bx+u(x), \,\forall x\in \mathbb{L}\,\text{for some $\mathbf{u}\in\mathcal{U}$}
\big\}.
\]

For a displacement $\mathbf{u}\in \mathcal{U}$ and its discrete directional derivatives, we employ the weighted
discrete $\ell_{\epsilon}^{2}$  and $\ell_{\epsilon}^{\infty}$ norms given by
\begin{align*}
&\|\mathbf{u}\|_{\ell_{\epsilon}^{2}}:= \left( \epsilon^2
\sum_{x\in\mathcal{L}}|u(x)|^{2}\right)^{1/2},\quad
\|\mathbf{u}\|_{\ell_{\epsilon}^{\infty}}:=\max\limits_{x\in\mathcal{L}}|u(x)|,\quad\text{and}\\
&\qquad\quad\|D\mathbf{u}\|_{\ell_{\epsilon}^{2}}:=
\left(\epsilon^2\sum_{x\in\mathcal{L}}\sum_{i=1}^{3}|D_{a_i}u(x)|^2\right)^{1/2}.
\end{align*}
The inner product associated with $\ell^2_\epsilon$ is
\[
\la \mathbf{u},\mathbf{w}\ra :=\epsilon^2\sum\limits_{x\in\mathcal{L}}u(x)\cdot w(x).
\]
\subsection{The B-QCF operator}
The total scaled atomistic energy for a periodic computational cell
$\Omega$ is
\begin{align}
\mathcal{E}^{a}(\mathbf{y})=&\frac{\epsilon^2}{2}\sum_{x\in\mathcal{L}}\sum_{r\in \mathcal{N}}
\phi(D_{r}y({x}))\label{AtomEnergy2D}
= {\epsilon^2}\sum_{x\in\mathcal{L}}\sum_{i=1}^{3}\big[\phi(D_{a_{i}}y({x}))
+\phi(D_{b_{i}}y({x}))\big],
\end{align}
where $\phi \in C^2(\mathbb{R}^2)$, for the sake of simplicity.
Typically, one assumes $\phi(r) = \varphi(|r|)$; the more general form
we use gives rise to a simplified notation; see also
\cite{OrtnerShapeev:2010}. We define $\phi'(r) \in \mathbb{R}^2$ and
$\phi''(r) \in \mathbb{R}^{2 \times 2}$ to be, respectively, the
gradient and hessian of $\phi$.

The equilibrium equations are given by the force balance at each atom,
\begin{equation}\label{AtomEquil2D}
F^{a}(x;y)+f(x;y)=0,\quad \text{for}\quad x\in\mathcal{L},
\end{equation}
where $f(x;y)$ are the external forces and $F^{a}(x;y)$ are the
atomistic forces (per unit volume $\epsilon^2$)
\begin{align*}
F^{a}(x;y):=&-\frac{1}{\epsilon^2}\frac{\partial \mathcal{E}^{a}(\mathbf{y})}{\partial y(x)}\\
=
 &- \frac{1}{\epsilon} \sum_{i=1}^{3}\Big[\phi'\left(D_{a_i}y(x)\right) +\phi'\left(D_{-a_i}y(x)\right)
           \Big]
 -\frac{1}{\epsilon} \sum_{i=1}^{3}\Big[\phi'\left(D_{b_i}y(x)\right) +\phi'\left(D_{-b_i}y(x)\right)
           \Big].
\end{align*}
Again, since $\mathbf{u} = \mathbf{y} - \mathbf{y}_B$, where $y_B(x) =
B x$, is assumed to be small we can linearize the atomistic
equilibrium equation \eqref{AtomEquil2D} about $\mathbf{y}_B$:
\[
\left(L^a\mathbf{u}^a\right)(x)=f(x),\quad \text{for}\quad x\in \mathcal{L},
\]
where $\left(L^a\mathbf{v}\right)(x)$, for a displacement $\mathbf{v}$, is given by
\[
\left(L^a\mathbf{v}\right)(x)=-\sum_{i=1}^{3}\phi''(Ba_{i})D_{a_{i}}D_{a_{i}}v(x-a_{i})
-\sum_{i=1}^{3}\phi''(Bb_{i})D_{b_{i}}D_{b_{i}}v(x-b_{i}),\quad\text{for}\quad x\in\mathcal{L}.
\]

The QCL approximation uses the Cauchy--Born extrapolation rule to
approximate the nonlocal atomistic model by a local continuum model
\cite{Ortiz:1995a,Shenoy:1999a,Miller:2003a}.  According to the bond
density lemma \cite[Lemma 3.2]{OrtnerShapeev:2010} (see also
\cite{Shapeev2D:2011}), we can write the total QCL energy as a sum of
the bond density integrals
\begin{equation}
  \label{QCLEnergy2D}
  \mathcal{E}^{c}(\mathbf{y})= \int_\Omega \sum_{r \in \mathcal{N}}
  \phi(\partial_r y) \, dx
  = \sum_{x\in \mathcal{L}}\sum_{r\in \mathcal{N}}
  \int_{0}^{1}\phi\big(\partial_{r}y(x+tr)\big)dt,
\end{equation}
where $\partial_{r} y(x) = \frac{d}{dt} y(x+t r)|_{t = 0}$ denotes the
directional derivative. We compute the continuum force $F^{c}(x;y) =
-\frac{1}{\epsilon^2} \frac{\partial\mathcal{E}^c}{\partial y(x)}$,
and linearize the force equation about the uniform deformation
$\mathbf{y}_{B}$ to obtain
\[
\left(L^{c}\mathbf{u}^{c}\right)(x)=f(x),\quad \text{for}\quad x\in \mathcal{L}.
\]

To formulate the B-QCF method, let the blending function
$\beta(s):\mathbb{R}^2\rightarrow [0, 1]$ be a "smooth",
$\Omega$-periodic function. We shall suppose throughout that
$R_a, R_b$ are chosen in such a way that
\begin{equation}
  \label{eq:supp_beta}
  {\rm supp}(D_{a_{i_1}} D_{a_{i_2}} D_{a_{i_3}} \beta) \subset
  \Omega_b \qquad \forall i \in \{1,\dots, 6\}^3.
\end{equation}
Then, the (nonlinear) B-QCF forces are given by
\begin{displaymath}
  F^{bqcf}(x; y) := \beta(x) F^a(x; y) + (1-\beta(x)) F^c(x; y),
\end{displaymath}
and linearizing the equilibrium equation $F^{bqcf} + f = 0$ about
$y_B$ yields
\begin{equation}
  \label{eq:2}
  \begin{split}
    & (L^{bqcf} \mathbf{u}^{bqcf})(x) = f(x), \quad \text{for } x \in
    \mathcal{L},\\
    &\text{where} \quad  (L^{bqcf} \mathbf{v})(x) = \beta(x) (L^a
    \mathbf{v})(x) + (1-\beta(x)) (L^c\mathbf{v})(x).
  \end{split}
\end{equation}

Since the nearest neighbor terms in the atomistic and the QCL models
are the same, we will focus on the second-neighbor interactions. We
rewrite the operator $L^{bqcf}$ in the form
\begin{align*}
  (L^{bqcf}\mathbf{v})(x) =~& \sum_{r \in \mathcal{N}}
  (L^{bqcf}_r \mathbf{v})(x), \\
  \text{where} \quad L^{bqcf}_r \mathbf{v}(x) =~& \beta(x) (L^a_r
  \mathbf{v})(x) + (1-\beta(x)) (L^c_r \mathbf{v})(x),
\end{align*}
where the nearest-neighbor operators are given by
\begin{displaymath}
  L^a_{a_j} \mathbf{v}(x) = L^c_{a_j} \mathbf{v}(x) = - \phi''(Ba_j) D_{a_j}
D_{a_j} v(x-a_j),
\end{displaymath}
and the second-neighbor operators, stated for convenience only for
$b_1 = a_1 + a_2$, by
\begin{align*}
\left(L^a_{b_1}\mathbf{u}\right)(x)=&
- \phi''(B b_1) D_{b_{1}}D_{b_{1}}v(x-b_{1}),\quad\text{while}\\
\left(L^{c}_{b_1}\mathbf{u}\right)(x)=&
- \phi''(Bb_1) \big[D_{a_1}D_{a_1}u(x-a_1)+
D_{a_2}D_{a_2}u(x-a_2)\big.\\
&\qquad\qquad\qquad \big.+D_{a_1}D_{a_2}u(x-a_1)+D_{a_1}D_{a_2}u(x-a_2)\big].
\end{align*}

\subsection{Auxiliary results}
The following is the 2D counterpart of the summation by parts
formula. The proof is straightforward.
\begin{lemma}[Summation by parts]
  For any $\mathbf{u} \in \mathcal{U}$ and any direction $r\in
  \mathbb{Z}^2$, we have
\begin{equation}\label{SumByPart2D}
\sum_{x\in\mathcal{L}}D_rD_r u(x-r)\cdot u(x)=-\sum_{x\in\mathcal{L}}D_{r}u(x-r)\cdot D_{r}u(x-r).
\end{equation}
\end{lemma}

The second auxiliary result we require is a trace- or Poincar\'e-type
inequality to bound 
$\|\mathbf{u}\|_{\ell^2_\epsilon(\Omega_b)}$
in terms of global norms. As a first step we establish a continuous
version of the inequality we are seeking. The key technical ingredient
in its proof is a sharp trace inequality, which is stated in Section
\ref{sec:app_trace}.

\begin{lemma}
  \label{th:blend_poincare_cts}
  Let $r_a < r_b \in (0, 1/2]$, and let $H := {\tt Hex}(r_b) \setminus
  {\tt Hex}(r_a)$; then there exists a constant $C$ that is
  independent of $r_a, r_b$ such that
  \begin{equation}
    \label{eq:blend_poincare_cts}
    \| u \|_{L^2(H)}^2 \leq C \big[(r_b - r_a) r_b |\log r_b |\big] \| \partial u
    \|_{L^2(\Omega)}^2
    \qquad \forall u \in H^1(\Omega), \int_\Omega u dx = 0.
  \end{equation}
\end{lemma}
\begin{proof}
  Let $\Sigma := \partial {\tt Hex}(1)$, and let $dS$ denote the
  surface measure, then
  \begin{displaymath}
    \| u \|_{L^2(H)}^2 = \int_{r = r_a}^{r_b} \int_{\Sigma} |u|^2 dS\,dr.
  \end{displaymath}
  Applying \eqref{eq:general_trace} with $r_0 = r$ and $r_1 = 1$ to
  each surface integral, we obtain
  \begin{displaymath}
    \| u \|_{L^2(H)}^2 \leq (r_b-r_a) \big( C_0 \| u \|_{L^2(\Omega)}^2 + C_1
    \| \partial u \|_{L^2(\Omega)}^2 \big),
  \end{displaymath}
  where $C_0 \leq 8 r_b$ and $C_1 = 2 r_b |\log r_b|$. An
  application of Poincar\'{e}'s inequality yields
  \eqref{eq:blend_poincare_cts}.
\end{proof}

In our analysis, we require a result as \eqref{eq:blend_poincare_cts}
for discrete norms. We establish this next, using straightforward
norm-equivalence arguments.

\def\CPab{C_P^{a,b}}
\begin{lemma}\label{DiscrBlenPoin:lemma}
  Suppose that $R_b \leq N/2$, then
  \begin{equation}
    \label{eq:blend_poincare}
    \| \mathbf{u} \|_{\ell^2_\epsilon(\mathcal{L}^b)}^2 \leq C\, (\CPab)^2 \| D \mathbf{u}
    \|_{\ell^2_\epsilon}^2
    \qquad \forall \mathbf{u} \in \mathcal{U}.
  \end{equation}
  where $C$ is a generic constant, and $\CPab := \big[ (\epsilon K)
  (\epsilon R_b) |\log (\epsilon R_b)| \big]^{1/2}$.
\end{lemma}
\begin{proof}
  Recall the identification of $\mathbf{u}$ with its corresponding
  $P_1$-interpolant.  Let $T \in \mathcal{T}$ with corners $x_j$, $j =
  1, 2, 3$, then
  \begin{displaymath}
    \int_T u \,dx = \frac{|T|}{3} \sum_{j = 1}^3 u(x_j), \quad
    \text{and hence} \quad
    \int_\Omega u\, dx = 0 \quad \forall \mathbf{u} \in \mathcal{U}.
  \end{displaymath}

  Let $r_a := \epsilon R_a$ and $r_b := \epsilon R_b$, then $H$
  defined in Lemma \ref{th:blend_poincare_cts} is identical to
  $\Omega_b$. For any element $T \subset \Omega_b$ it is
  straightforward to show that
  \begin{displaymath}
    \| \mathbf{u} \|_{\ell^2_\epsilon(T)} \leq C \| u \|_{L^2(T)}.
  \end{displaymath}
  This immediately implies
  \begin{equation}
    \label{eq:blend_ineq:10}
    \| \mathbf{u} \|_{\ell^2_\epsilon(\mathcal{L}^b)} \leq C \| u \|_{L^2(H)},
  \end{equation}
  for a constant $C$ that is independent of $\epsilon$, $R_a$, $K$ and
  $\mathbf{u}$.  Applying \eqref{eq:blend_poincare_cts} yields
  \begin{displaymath}
    \| \mathbf{u} \|_{\ell^2_\epsilon(\mathcal{L}^b)}^2 \leq C \big[
    (r_b-r_a)
    r_b |\log r_b| \big] \|\partial u \|_{L^2(\Omega)}^2.
  \end{displaymath}

  Fix $T \in \mathcal{T}$ and let $x_j \in T$ such that $x_j + a_j \in
  T$ as well. Employing \cite[Eq. (2.1)]{OrtnerShapeev:2010} we obtain
  \begin{displaymath}
    \sum_{j = 1}^3 \big| D_{a_j} u(x_j) \big|^2 = \sum_{j = 1}^3
    \big| (\partial u|_T) a_j \big|^2 = \smfrac{3}{2} \big| \partial u|_T \big|^2,
  \end{displaymath}
  and summing over $T \in \mathcal{T}, T \subset \bar{\Omega}$ we
  obtain that $\| \partial u \|_{L^2(\Omega)} \leq C \| D \mathbf{u}
  \|_{\ell^2_\epsilon}$. This concludes the proof.
\end{proof}

\subsection{Bounds on $L^{bqcf}_{b_1}$}
We focus only on the $b_1$-bonds, however, by symmetry analogous
results hold for all second-neighbor bonds. As in the 1D case, we
begin by converting the quadratic form $\la
L^{bqcf}_{b_1}\mathbf{u},\mathbf{u}\ra$ into divergence form. To that
end it is convenient to define the bond-dependent symmetric bilinear
forms and quadratic forms (although we write them like a norm they are
typically indefinite)
\begin{align*}
  \la r, s \ra_{b} := r^{\rm T} \phi''(Bb) s, \quad \text{and} \quad
  |r|_{b}^2 := \la r, r \ra_{b}, \qquad \text{for } r, s, b \in \mathbb{R}^2.
\end{align*}

\begin{lemma}\label{DivformLemma2D}
For any displacement $\mathbf{u}\in \mathcal{U}$, we have
\begin{equation}\label{Divform2D}
\la L^{bqcf}_{b_1}\mathbf{u},\mathbf{u}\ra
= \la L^{c}_{b_1}\mathbf{u},\mathbf{u}\ra-\epsilon^4\sum_{x\in\mathcal{L}}\beta(x-a_2)|D_{a_1}D_{a_2}u(x-a_1-a_2)|_{b_1}^2
+\mathbf{R}_{b_1}+\mathbf{S}_{b_1},
\end{equation}
where
\begin{align}\label{SigPart2D}
\begin{split}
\mathbf{R}_{b_1}:=&
-\epsilon^4\sum_{x\in\mathcal{L}}\big\{ D_{a_1}\beta(x-2a_1)
 \big\la D_{a_1}u(x-2a_1),
 D_{a_2}D_{a_2}u(x-a_1-a_2) \big\ra_{b_1}\\
&\qquad\qquad\quad
+ D_{a_2}\beta(x-a_2) \big\la D_{a_1}u(x-a_1),
D_{a_1}D_{a_2}u(x-a_1-a_2) \big\ra_{b_1}
\big\},\quad\text{and}\\
\mathbf{S}_{b_1}:=&-\epsilon^4\sum_{x\in\mathcal{L}}
D_{a_1}D_{a_1}\beta(x-2a_1)
\big\la u(x-a_1), D_{a_2}D_{a_2}u(x-a_1-a_2) \big\ra_{b_1}.
\end{split}
\end{align}
\end{lemma}
\begin{proof}
  For this purely algebraic proof we may assume without loss of
  generality that $\phi''(Bb_1) = {\rm I}$. In general, one may simply
  replace all Euclidean inner products with $\la \cdot, \cdot
  \ra_{b_1}$.

Starting from \eqref{eq:2}, we have
\begin{align*}
\la L^{bqcf}_{b_1}\mathbf{u},\mathbf{u}\ra
=& \la L^{c}_{b_1}\mathbf{u},\mathbf{u}\ra+\la L^{a}_{b_1}\mathbf{u}- L^{c}_{b_1}\mathbf{u},\beta\mathbf{u}\ra \\
=& \la L^{c}_{b_1}\mathbf{u},\mathbf{u}\ra
-\epsilon^2\sum_{x\in\mathcal{L}}\beta(x)u(x)\cdot\left[D_{b_1}D_{b_1}u(x-b_1)-D_{a_1}D_{a_1}u(x-a_1)
\right.\\
&\qquad\qquad\left. -D_{a_2}D_{a_2}u(x-a_2) -D_{a_1}D_{a_2}u(x-a_1)-D_{a_1}D_{a_2}u(x-a_2)\right].
\end{align*}
We will focus our analysis on $\la L^{a}_{b_1}\mathbf{u}- L^{c}_{b_1}\mathbf{u},\beta\mathbf{u}\ra$.

Noting that $b_1=a_1+a_2$, one can recast $D_{b_1}D_{b_1}u(x-b_1)$ as
\begin{align*}
D_{b_1}&D_{b_1}u(x-b_1)\\
=&\frac{1}{\epsilon^2}\left[u(x+b_1)-2u(x)+u(x-b_{1})\right]\\
=&D_{a_1}D_{a_2}u(x)+ D_{a_1}D_{a_1}u(x-a_1)
+D_{a_2}D_{a_2}u(x-a_2)+D_{a_1}D_{a_2}u(x-a_1-a_2).
\end{align*}
Applying the summation by parts formula \eqref{SumByPart2D}
to $\la L^{a}_{b_1}\mathbf{u}- L^{c}_{b_1}\mathbf{u},\beta\mathbf{u}\ra$, we
get
\begin{align*}
\la L^{a}_{b_1}\mathbf{u}- L^{c}_{b_1}\mathbf{u},\beta\mathbf{u}\ra
=& -\epsilon^3 \sum_{x\in\mathcal{L}}\beta(x)u(x)\cdot \Big[D_{a_{1}}D_{a_{1}}D_{a_{2}}u(x-a_{1})-D_{a_{1}}D_{a_{1}}D_{a_{2}}u(x-a_1-a_2)\Big]\\
=&  -\epsilon^4 \sum_{x\in\mathcal{L}}\beta(x)u(x)\cdot D_{a_{1}}D_{a_{1}}D_{a_{2}}D_{a_2}u(x-a_1-a_2)\\
=& \epsilon^4\sum_{x\in\mathcal{L}}D_{a_1}D_{a_2}D_{a_2}u(x-a_1-a_2)\cdot D_{a_1}\Big(\beta(x-a_1)u(x-a_1)\Big)\\
=&  \epsilon^4\sum_{x\in\mathcal{L}}D_{a_1}D_{a_2}D_{a_2}u(x-a_1-a_2)\cdot \Big[\beta(x)D_{a_1}u(x-a_1)+u(x-a_1)D_{a_1}\beta(x-a_1)\Big].
\end{align*}
Another application of the summation by parts formula
\eqref{SumByPart2D} converts $\la L^{a}_{b_1}\mathbf{u}-
L^{c}_{b_1}\mathbf{u},\beta\mathbf{u}\ra$ into
\begin{align*}
\la L^{a}_{b_1}\mathbf{u}- L^{c}_{b_1}\mathbf{u},\beta\mathbf{u}\ra
=&  \epsilon^4\sum_{x\in\mathcal{L}}D_{a_1}D_{a_2}D_{a_2}u(x-a_1-a_2)\cdot \big(u(x-a_1)D_{a_1}\beta(x-a_1)\big)\\
&\quad -\epsilon^4\sum_{x\in\mathcal{L}}D_{a_1}D_{a_2}u(x-a_1-a_2)\cdot \big(D_{a_2}\beta(x-a_2)D_{a_1}u(x-a_1)\big)\\
&\quad\quad-  \epsilon^4\sum_{x\in\mathcal{L}}D_{a_1}D_{a_2}u(x-a_1-a_2)\cdot\big( \beta(x-a_2)D_{a_1}D_{a_2}u(x-a_1-a_2)\big).
\end{align*}
The first two terms on the right-hand side can be rewritten as
\begin{align*}
\epsilon^4&\sum_{x\in\mathcal{L}}\big\{D_{a_1}D_{a_2}D_{a_2}u(x-a_1-a_2)\cdot \big(u(x-a_1)D_{a_1}\beta(x-a_1)\big)\big.\\
&\qquad\Big.-D_{a_1}D_{a_2}u(x-a_1-a_2)\cdot \big(D_{a_2}\beta(x-a_2)D_{a_1}u(x-a_1)\big)\big\}\\
&=-\epsilon^4\sum_{x\in\mathcal{L}}\Big(u(x-a_1)D_{a_1}D_{a_1}\beta(x-2a_1)\Big)\cdot D_{a_2}D_{a_2}u(x-a_1-a_2) \\
&\qquad-\epsilon^4\sum_{x\in\mathcal{L}}\big\{
D_{a_1}\beta(x-2a_1)D_{a_1}u(x-2a_1)\cdot D_{a_2}D_{a_2}u(x-a_1-a_2)\big.\\
&\qquad\qquad\qquad \big.
+D_{a_2}\beta(x-a_2)D_{a_1}u(x-a_1)\cdot D_{a_1}D_{a_2}u(x-a_1-a_2)
\big\}\\
&=\mathbf{S}_{b_1}+\mathbf{R}_{b_1}.
\end{align*}
Thus, we obtain \eqref{Divform2D} and \eqref{SigPart2D}.
\end{proof}

Next, we will bound the singular terms $\mathbf{R}_{b_1}$ and
$\mathbf{S}_{b_1}$, for which we introduce the notation
\begin{displaymath}
  \| D^{(2)} \beta \|_{\ell^\infty_\epsilon} := \max_{1 \leq i, j \leq
    6}  \| D_{a_i} D_{a_j} \beta \|_{\ell^\infty_\epsilon}, \quad
  \text{and} \quad
  \| D^{(3)} \beta \|_{\ell^\infty_\epsilon} := \max_{1 \leq i, j, k \leq
    6}  \| D_{a_i} D_{a_j} D_{a_k} \beta \|_{\ell^\infty_\epsilon}.
\end{displaymath}


\begin{lemma}\label{SigEstLemma2D}
  The terms $\mathbf{R}_{b_1}$ and $\mathbf{S}_{b_1}$ defined in
  \eqref{SigPart2D} are bounded by
\begin{align}
\left|\mathbf{R}_{b_1}\right|
\le & 4\epsilon^2|\phi''(Bb_1)| \, \|D\beta\|_{\ell_\epsilon^{\infty}}\|D\mathbf{u}\|_{\ell_\epsilon^{2}}^2, \quad\text{and}\label{SigEst2D}\\
\left|\mathbf{S}_{b_1}\right|
\le&
C\epsilon^{2} |\phi''(Bb_1)| \Big[ \|D^{(2)} \beta\|_{\ell_{\epsilon}^{\infty}}
+\|D^{(3)} \beta\|_{\ell_{\epsilon}^{\infty}}\, \CPab \Big]\|D\mathbf{u}\|_{\ell_{\epsilon}^2}^2,
\end{align}
where $C$ is a generic constant and $\CPab$ is defined in Lemma \ref{eq:blend_poincare}.
\end{lemma}
\begin{proof}
According to the expression of $\mathbf{R}_{b_1}$ given in \eqref{SigPart2D}
and noting that
\[
\|D_{a_2}D_{a_2}\mathbf{u}\|_{\ell_\epsilon^2}^2\le
\frac{4}{\epsilon^2}\|D\mathbf{u}\|_{\ell_\epsilon^2}^2
\quad \text{and} \quad  \|D_{a_1}D_{a_2}\mathbf{u}\|_{\ell_\epsilon^2}^2\le
\frac{4}{\epsilon^2}\|D\mathbf{u}\|_{\ell_\epsilon^2}^2,
\]
we immediately obtain the first inequality of \eqref{SigEst2D}.

We first rewrite $\mathbf{S}_{b_1}$ as
\begin{align}
\mathbf{S}_{b_1}=&
-\epsilon^4\sum_{x\in\mathcal{L}}D_{a_1}D_{a_1}\beta(x-2a_1) \big\la
D_{a_2}D_{a_2}u(x-a_1-a_2),  u(x-a_1) \big\ra_{b_1} \nonumber\\
=&-\epsilon^4\sum_{x\in\mathcal{L}}D_{a_1}D_{a_1}\beta(x-2a_1)
D_{a_2}\big\la
D_{a_2}u(x-a_1-a_2), u(x-a_1-a_2)\big\ra_{b_1} \nonumber\\
&\qquad+ \epsilon^4\sum_{x\in\mathcal{L}}D_{a_1}D_{a_1}\beta(x-2a_1)
\big\la D_{a_2}u(x-a_1-a_2), D_{a_2}u(x-a_1-a_2)\big\ra_{b_1}\nonumber\\
=&
\epsilon^4\sum_{x\in\mathcal{L}}D_{a_2}D_{a_1}D_{a_1}\beta(x-2a_1-a_2)
\big\la D_{a_2}u(x-a_1-a_2)\cdot u(x-a_1-a_2) \big\ra_{b_1} \nonumber\\
&\qquad+
\epsilon^4\sum_{x\in\mathcal{L}}D_{a_1}D_{a_1}\beta(x-2a_1)\big|D_{a_2}u(x-a_1-a_2)\big|_{b_1}^2.\label{SigEst2DEq1}
\end{align}
For the second term in \eqref{SigEst2DEq1}, we
have
\[
\Big| \epsilon^4\sum_{x\in\mathcal{L}}D_{a_1}D_{a_1}\beta(x-2a_1)\big|D_{a_2}u(x-a_1-a_2)\big|_{b_1}^2 \Big|
\le
\epsilon^2 | \phi''(Bb_1)| \|D^{(2)} \beta\|_{\ell_{\epsilon}^{\infty}}\|D\mathbf{u}\|_{\ell_{\epsilon}^2}^2.
\]
For the first term, we have
\begin{align*}
&\Big|\epsilon^4\sum_{x\in\mathcal{L}}D_{a_2}D_{a_1}D_{a_1}\beta(x-2a_1-a_2)
\big\la D_{a_2}u(x-a_1-a_2), u(x-a_1-a_2)\big\ra_{b_1} \Big|\\
&\qquad\le \epsilon^2 \,|\phi''(Bb_1)|\,\|\mathbf{u}\,D_{a_2}D_{a_1}D_{a_1}\beta\|_{\ell_{\epsilon}^2} \|D\mathbf{u}\|_{\ell_{\epsilon}^2}
\le \epsilon^2 \,|\phi''(Bb_1)| \,\|D^{(3)} \beta\|_{\ell_{\epsilon}^{\infty}}
\,\|\mathbf{u}\|_{\ell_{\epsilon}^2(\mathcal{L}^b)} \,\|D\mathbf{u}\|_{\ell_{\epsilon}^2}.
\end{align*}
The last inequality comes from the assumption \eqref{eq:supp_beta},
which ensures that ${\rm supp}(D_{a_2}D_{a_1}D_{a_1}\beta) \subset
\Omega_b$.

Applying Lemma~\ref{DiscrBlenPoin:lemma} yields the bound for ${\bf
  S}_{b_1}$.
\end{proof}

To summarize the estimates of this section we define a self-adjoint
operator $\tilde{L}$ by
\begin{equation}
  \label{eq:defn_Ltil}
  \la \tilde{L} \mathbf{u}, \mathbf{u} \ra :=
  \la L^c \mathbf{u}, \mathbf{u} \ra - \epsilon^4 \sum_{j = 1}^3
  \sum_{x \in \mathcal{L}} \beta(x-a_2) \big| D_{a_j}
  D_{a_{j+1}} u(x - a_1 - a_2) \big|_{b_1}^2;
\end{equation}
then, Lemma \ref{DivformLemma2D} and Lemma \ref{SigEstLemma2D}
immediately yield the following result.

\begin{corollary}
  \label{th:lbqcf_Ltil_bound}
  Suppose that $R_a$ and $R_b$ are defined such that
  \eqref{eq:supp_beta} holds; then, for all $\mathbf{u} \in
  \mathcal{U}$,
  \begin{equation}
    \label{eq:Lbqcf_Ltil_bound}
    \la L^{bqcf} \mathbf{u}, \mathbf{u} \ra \geq \la \tilde{L}
    \mathbf{u}, \mathbf{u} \ra - C \, C'' \big[ \epsilon^2 \| D \beta
    \|_{\ell^\infty} + \epsilon^2 \|D^{(2)}\beta\|_{\ell^\infty}  +
    \epsilon^2 \CPab \| D^{(3)} \beta \|_{\ell^\infty} \big]
    \|D\mathbf{u}\|_{\ell^2_\epsilon}^2,
  \end{equation}
  where $C$ is a generic constant, $C'' := \max_{j = 1,2,3} |\phi''(B
  b_j)|$ and $\CPab$ is defined in Lemma \ref{eq:blend_poincare}.
\end{corollary}

Based on the analysis and numerical experiments in
\cite{OrtnerShapeev:2010} for a similar linearized operator we expect
that the region of stability for $\tilde{L}$ is the same as for $L^a$;
that is, $\tilde{L}$ is positive definite for a macroscopic strain $B$
if and only if $L^a$ is positive definite. However, we are at this
point unable to prove this result. Instead, we have the following
weaker result. The proof is elementary.

\begin{proposition}
  \label{th:stab_Ltil}
  Suppose that $B \in \mathbb{R}^{2 \times 2}$ is such that $L^c$ is
  positive definite,
  \begin{displaymath}
    \la L^c \mathbf{u}, \mathbf{u} \ra \geq \gamma_c \| D \mathbf{u}
    \|_{\ell^2_\epsilon}^2 \qquad \forall \mathbf{u} \in \mathcal{U},
  \end{displaymath}
  and suppose that $\phi''(Bb_j) \leq \delta {\rm I}$ where $\delta <
  \gamma_c / 4$, then $\tilde{L}$ is positive definite,
  \begin{equation}
    \label{eq:stab_Ltil}
    \la \tilde{L} \mathbf{u}, \mathbf{u} \ra \geq \tilde{\gamma} \| D \mathbf{u}
    \|_{\ell^2_\epsilon}^2 \qquad \forall \mathbf{u} \in \mathcal{U},
  \end{equation}
  with $\tilde{\gamma} = \gamma_c - 4\delta$.
\end{proposition}

\subsection{Positivity of the B-QCF operator in 2D}
The \textit{blending width} $K$ is again a crucial ingredient in the
stability analysis for $L^{bqcf}$. Due to the simple geometry we have
chosen it straightforward to generalize Lemma~\ref{BlendFunEstLemma}
to the two-dimensional case, using the same arguments as in 1D.

\begin{lemma}
  \label{2DBlendFunEstLemma}
  It is possible to choose $\beta$ such that
  \begin{equation}
    \label{2Deq:BlendFunEst_upper}
    \| D^{(j)} \beta \|_{\ell^\infty} \leq C_\beta (K \epsilon)^{-j}.
    \quad \text{for } j = 1, 2, 3,
  \end{equation}
\end{lemma}

Since we cannot fully characterize the stability of $\tilde{L}$ in
terms of the stability of $L^a$ or $L^c$ we will only prove stability
of $L^{bqcf}$ subject to the assumption that $\tilde{L}$ is
stable. Proposition \ref{th:stab_Ltil} gives sufficient conditions.

\begin{theorem}\label{BlendSizeThm2D}
  Suppose that $\beta$ is chosen quasi-optimally so that
  \eqref{2Deq:BlendFunEst_upper} is attained; then,
  \begin{displaymath}
    \la L^{bqcf} \mathbf{u}, \mathbf{u} \ra \geq \gamma_{bqcf}
    \| D\mathbf{u} \|_{\ell^2_\epsilon}^2,
  \end{displaymath}
  where
  \begin{displaymath}
    \gamma_{bqcf} := \tilde{\gamma} -
    C\, C''\, \big[ \epsilon^{-1/2} K^{-5/2} |\epsilon R_b
    \log(\epsilon R_b)|^{1/2} \big],
  \end{displaymath}
  where $C$ is a generic constant and $C''$ is defined in Corollary
  \ref{th:lbqcf_Ltil_bound}.

  In particular, if $\tilde{L}$ is positive definite
  \eqref{eq:stab_Ltil} and if $K$ is sufficiently large, then
  $L^{bqcf}$ is positive definite.
\end{theorem}
\begin{proof}
  From Corollary \ref{th:lbqcf_Ltil_bound} and
  \eqref{2Deq:BlendFunEst_upper} we obtain
  \begin{align*}
    \la L^{bqcf} \mathbf{u}, \mathbf{u} \ra  \geq~&
    \big\{ \tilde{\gamma} - C\,C''\,\big[ \epsilon^2 (\epsilon K)^{-1}
    + \epsilon^2 (\epsilon K)^{-2} + \epsilon^2 (\epsilon K)^{-5/2}
     |\epsilon R_b\log(\epsilon R_b)|^{1/2} \big]\big\}
    \| D \mathbf{u} \|_{\ell^2_\epsilon}^2 \\
    \geq~& \big\{ \tilde{\gamma} - C\,C''\,\big[ \epsilon^{-1/2} K^{-5/2}
   |\epsilon R_b \log(\epsilon R_b)|^{1/2} \big] \big\}
    \| D \mathbf{u} \|_{\ell^2_\epsilon}^2. \qedhere
  \end{align*}
\end{proof}

\begin{remark}
  Suppose that $\tilde{\gamma} > 0$, uniformly as $N \to \infty$ (or,
  $\epsilon \to 0$). In this limit, we would like to understand how to
  optimally scale $K$ with $R_a$. (Note that $R_a$ controls the
  modeling error; cf. Remark \ref{2DB-QCFremark2}.) We consider three
  different scalings of $R_a$.

 {\it Case 1: } Suppose that $R_a$ is bounded as $\epsilon \to 0$. In
  that case, we obtain
  \begin{align}
    \notag
    \gamma_{bqcf} - \tilde{\gamma} =~& - C \,C''\,
    \epsilon^{-1/2} K^{-5/2} | \epsilon (R_a+K) \log (\epsilon
    (R_a+K)) |^{1/2} \\
    \notag
    =~& - C \,C''\, K^{-2} \big| \big(1 + \smfrac{R_a}{K}\big)
    \big(\log (\epsilon K) + \log(1+\smfrac{R_a}{K}) \big)\big|^{1/2}
    \\
    \label{eq:scaling_case2}
    \eqsim~& - C\, C''\, K^{-2} |\log(\epsilon K)|^{1/2}.
  \end{align}
  From this it is easy to see that $L^{bqcf}$ will be positive
  definite provided we select $K \gg |\log \epsilon|^{1/4}$.

  {\it Case 2: } Suppose that $1 \ll R_a \ll \epsilon^{-1}$; to
  precise, let $R_a \sim \epsilon^{-\alpha}$ for some $\alpha \in (0,
  1)$. Then, a similar computation as \eqref{eq:scaling_case2} yields
  \begin{displaymath}
    \gamma_{bqcf} - \tilde{\gamma} \eqsim K^{-5/2} \big| (K +
    \epsilon^{-\alpha}) (\log\epsilon + \log(K+\epsilon^{-\alpha}))\big|^{1/2},
  \end{displaymath}
  and we deduce that, in this case, $L^{bqcf}$ will positive definite
  provided we select $K \gg \epsilon^{-\alpha/5} |\log
  \epsilon|^{1/5}$.

  {\it Case 3: } Finally, the case when the atomistic region is
  macroscopic, i.e., $R_a = O(\epsilon^{-1})$, can be treated
  precisely as the 1D case and hence we obtain that, if we select $K
  \gg \epsilon^{-1/5}$, then $L^{bqcf}$ is positive.

  In summary, we have shown that, in the limit as $\epsilon \to 0$, if
  $\tilde{L}$ is positive definite, $R_a = O(\epsilon^{-\alpha})$ and
  if we choose
  \begin{align}\label{BlendSize2D}
    K\gg
    \begin{cases}
      |\log{\epsilon}|^{1/4}, & \quad \alpha = 0,  \\
      |\log\epsilon|^{1/5}\epsilon^{-\alpha/5},& \quad 0<\alpha<1,\\
      \epsilon^{-1/5}, & \quad \alpha = 1,
    \end{cases}
  \end{align}
  then the B-QCF operator $L^{bqcf}$ is positive definite and
  $\gamma^{bqcf} \sim \tilde{\gamma}$ as $\epsilon \to 0$. We
  emphasize that these are very mild restrictions on the blending
  width. \hfill \qed
\end{remark}

It remains to show that the sufficient conditions we derived to
guarantee positivity of $L^{bqcf}$ are sharp. A result as general as
\eqref{eq:1d_coerc_upper} in 1D would be very technical to obtain;
instead, we offer a brief formal discussion for a special case.

\begin{figure}
  \includegraphics[width=7cm]{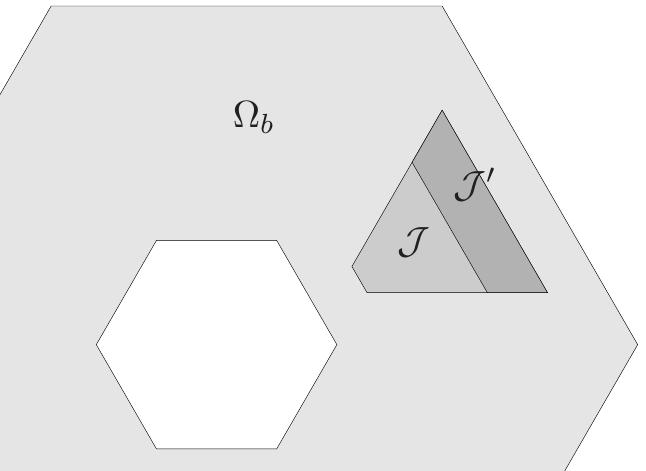}
  \caption{\label{fig:sharp_2d} Visualization of the construction
    discussed in \ref{rem:sharp_2d}: the white region is the atomistic
    domain, the light gray region the blending region, the medium gray
    region and dark gray regions together are the set $\mathcal{J}$
    and the dark gray region is the set $\mathcal{J}'$.}
\end{figure}

\begin{remark} \label{rem:sharp_2d} We consider again the limit as
  $\epsilon \to 0$, and for simplicity restrict ourselves to the case
  where $0 \ll K \eqsim \epsilon^{-\theta}$ and $0 \ll R_a \eqsim
  \epsilon^{-\alpha}$, for $0 < \theta \leq \alpha \leq 1$. In
  particular, $R_b \eqsim \epsilon^{-\alpha}$ as well.

  We assume that $D_{a_3} \beta(x) = 0$ for all $x \in \mathcal{J}
  \subset \mathcal{L}^b$, as depicted in Figure
  \ref{fig:sharp_2d}. The set $\mathcal{J}$ should be chosen so that
  its size is comparable with that of $\mathcal{L}^b$, but
  sufficiently small to still allow $\beta$ to satisfy the bound
  \eqref{2Deq:BlendFunEst_upper}. We can now repeat the 1D argument
  along atomic layers to obtain that
  \begin{displaymath}
    D_{a_2} D_{a_1} D_{a_1} \beta(x) \leq - \smfrac{1}{2} (\epsilon
    K)^{-3}
    \eqsim - \epsilon^{-3 + 3 \theta}
  \end{displaymath}
  for all $x$ in a subset $\mathcal{J}' \subset \mathcal{J}$
  containing entire atomic planes, that has comparable size to
  $\mathcal{J}$; that is, $\# \mathcal{J}' \eqsim K R_b \eqsim
  \epsilon^{-\theta-\alpha}$.

  Suppose now that $\phi''(Bb_1)$ has a negative eigenvalue $\lambda$
  with corresponding normalized eigenvector $\hat{u} \in
  \mathbb{R}^2$, then we seek test functions of the form $u(x) =
  \mu(x) \hat{u}$. It is now relatively straightforward, applying the
  1D argument in normal direction and using a smooth cut-off in the
  tangential direction, to construct $\mu$ supported in $\mathcal{J}'$
  with $D_{a_2} \mu(x) \eqsim (\epsilon^2 \# \mathcal{J}')^{-1/2}$ so
  that $\| D \mathbf{u} \|_{\ell^2_\epsilon} \eqsim 1$, and
  \begin{align*}
    &
    \epsilon^4\sum_{x\in\mathcal{L}}D_{a_2}D_{a_1}D_{a_1}\beta(x-2a_1-a_2)
    \big\la D_{a_2}u(x-a_1-a_2),  u(x-a_1-a_2) \big\ra_{b_1}   \\
    & \hspace{1cm} = \epsilon^4 \lambda_1
    \sum_{x\in\mathcal{L}}D_{a_2}D_{a_1}D_{a_1}\beta(x-2a_1-a_2)
    D_{a_2}\mu(x-a_1-a_2) \mu(x-a_1-a_2)  \\
    & \hspace{1cm} \lesssim - \epsilon^4 \lambda_1 (\# \mathcal{J}')
    (K \epsilon)^{-3} (\epsilon^2 \# \mathcal{J}')^{-1/2} \eqsim -
    \epsilon^{(5\theta-\alpha)/2}.
  \end{align*}
  This shows that, if $K \ll \epsilon^{-\alpha/5}$, then $L^{bqcf}$ is
  necessarily indefinite.

  In summary, for the specific interface geometry and a particular
  choice of $\beta$ (which does, however, lead to the quasi-optimal
  bound \eqref{2Deq:BlendFunEst_upper}) we have shown that Theorem
  \ref{BlendSizeThm2D} is sharp up to logarithmic terms. \hfill \qed
\end{remark}

\begin{remark}\label{2DB-QCFremark2}
  In practise, for the computation of different types of defects, we
  would first choose an appropriate scaling $R_a=\epsilon^{-\alpha}$
  for the atomistic region, considering the accuracy of the B-QCF
  method, and then choose the blending width $K$ in order to ensure
  stability.

  For instance, for a point defect in 2D with zero Burger's vector it
  is expected that the displacement field satisfies $u_a(x) = y_a(x) -
  Bx \eqsim \epsilon / r$, where $r$ is the distance from the defect
  \cite{PointDefect, OrtnerShapeev:2010}. Without coarse-graining, the
  local continuum (QCL) model has a modeling error of order
  $O(\epsilon^2 |\partial^3 u_a|)$ (see \cite{ortner:qnl1d,
    LuskinXingjie.qnl1d, Dobson:2008b} for proofs in 1D and
  \cite{VKOr:blend2} for a proof in arbitrary dimensions); and
  although we have not established it rigorously, we expect that
  modeling error for the B-QCF method outside the atomistic region is
  also of second order; see also \cite{qcf.stab}.

  From $u(x) \eqsim \epsilon / r$ we can make the reasonable
  assumption that $|\partial^3 y_a| \eqsim \epsilon / r^{4}$, from
  which we obtain (assuming also stability) that the total error is of
  the order
  \begin{displaymath}
    \| \partial (y_a - y_{bqcf}) \|_{L^2} \eqsim \epsilon^2
    \| \partial^3 y_a \|_{L^2(\Omega \setminus\Omega_a)} \eqsim
    \epsilon^3 \Big( \int_{\epsilon R_a}^1 r |r^{-4}|^2 dr \Big)^{1/2}
    \eqsim R_a^{-3}.
  \end{displaymath}
  Hence, if we wish to obtain $\| \partial (y_a - y_{bqcf}) \|_{L^2}
  \eqsim \epsilon^k$, $0 < k < 3$, then we need to choose
  \begin{displaymath}
    R_a \eqsim \epsilon^{-k/3}, \quad \text{and
      consequently} \quad K \gg \epsilon^{-k/15} |\log
    \epsilon|^{1/5}.
  \end{displaymath}
  With this choice we can ensure both the stability and
  $O(\epsilon^k)$ accuracy of the B-QCF method; provided that our
  assumption that the B-QCF method has indeed a second-order modelling
  error is correct. \hfill \qed
\end{remark}

\section{Conclusion}
We have studied the stability a blended force-based quasicontinuum
(B-QCF) method. In 1D we were able to identify an asymptotically
optimal condition on the width of the blending region to ensure that
the linearized B-QCF operator is coercive if and only if the atomistic
operator is coercive as well.  In the $2$D B-QCF model, we have
obtained rigorous sufficient conditions and have presented a heuristic
argument suggesting that they are sharp up to logarithmic terms. In 2D
our proof of coercivity of $L^{bqcf}$ relies on the coercivity of the
auxiliary operator $\tilde{L}$ defined in \eqref{eq:defn_Ltil}, for
which we cannot give sharp conditions at this point.

The main conclusion of this work is that the required blending width
to ensure coercivity of the linearized B-QCF operator is surprisingly
small.

Our analysis in this paper is the first step towards a complete a
priori error analysis of the B-QCF method, which will require a
coercivity analysis of the B-QCF operator linearized about arbitrary
states, as well as a consistency analysis in negative Sobolev norms.

\section{Appendix: A Trace Inequality}
\label{sec:app_trace}
In the following trace theorem, $S(1)$ denotes the unit sphere in
$\mathbb{R}^d$, $r := |x|$ and $\theta := x / |x|$. Upon taking $\psi
\equiv 1$ and employing standard orthogonal decompositions it is easy
to check that the result is sharp. In particular, for $d = 2$,
consider the case $u(x) = \log |x|$.

\begin{lemma}
  Let $d \geq 2$, $\psi : S(1) \to (0, 1]$ be Lipschitz continuous,
  and $\Sigma := \{ \psi(\sigma) \sigma : \sigma \in S(1)
  \}$. Moreover, let $0 < r_0 < r_1 \leq 1$, and $A := \bigcup_{r_0 <
    r < r_1} (r\Sigma)$, then
  \begin{align}
    \label{eq:general_trace}
    & \| u \|_{L^2(r_0 \Sigma)}^2 \leq C_0\| u \|_{L^2(A)}^2 +
    C_1 \| \partial u \|_{L^2(A)}^2, \quad \forall u \in H^1(A), \\
    & \text{where} \quad
    C_0 = \frac{2 d}{r_1 - r_0} \Big(\frac{r_0}{r_1}\Big)^{d-1},
    \quad \text{and}   \quad
    C_1 =
    \begin{cases}
      2 r_0 |\log r_0|, & d = 2 \\
      2 r_0 / (d-2), & d \geq 3.
    \end{cases}
  \end{align}
\end{lemma}
\begin{proof}
  Since $A$ is a Lipschitz domain we may assume, without loss of
  generality that $u \in C^1(\bar{A})$. The symbol $dS$ denotes the
  $(d-1)$-dimensional Hausdorff measure in $\mathbb{R}^d$.

  Let $r_0 < s < r_1$, then
  \begin{align}
    \notag
    \int_{r_0\Sigma} |u|^2 dS =~& r_0^{d-1} \int_\Sigma
    |u(r_0\sigma)|^2 dS_\sigma \\
    \notag
    =~& r_0^{d-1}  \int_\Sigma \bigg| u(s\sigma) -  \int_{r = r_0}^s \frac{d}{dr}
    u(r \sigma) dr \bigg|^2 dS_\sigma \\
    \label{eq:trace:10}
    \leq~& 2 r_0^{d-1} \int_\Sigma \big|u(s\sigma)\big|^2 dS_\sigma
    + 2 r_0^{d-1} \int_\Sigma \bigg| \int_{r = r_0}^s \partial u
    \cdot \sigma
    dr \bigg|^2 dS_\sigma.
  \end{align}
  By hypothesis we have $|\sigma| \leq 1$ for all $\sigma \in \Sigma$,
  hence the second term on the right-hand side can be further
  estimated, applying also the Cauchy--Schwartz inequality, by
  \begin{align*}
    2 r_0^{d-1} \int_\Sigma \bigg| \int_{r = r_0}^s \partial u
    \cdot \sigma
    dr \bigg|^2 dS_\sigma
    \leq~& 2 r_0^{d-1} \int_\Sigma \int_{r = r_0}^s r^{-d+1} dr \,
    \int_{r = r_0}^s r^{d-1} |\partial u(r\sigma)|^2 dr \, dS_\sigma \\
    =~& 2 r_0^{d-1} (J(s) - J(r_0)) \int_{r = r_0}^{s}
    \int_{r\Sigma} |\partial u|^2 dS \,dr \\
    \leq~& 2 r_0^{d-1} (J(s) - J(r_0)) \| \partial u \|_{L^2(A)}^2,
  \end{align*}
  where $J'(t) = t^{-d+1}$, that is, $J(t) = \log t$ if $d = 2$ and
  $J(t) = t^{-d+2} / (-d+2)$ if $d \geq 3$.  Since $J(s)$ is negative
  and strictly increasing for $s \leq 1$ we obtain
  \begin{equation}
    \label{eq:trace:20}
    2 r_0^{d-1} \int_\Sigma \bigg| \int_{r = r_0}^s \partial u
    \cdot \sigma
    dr \bigg|^2 dS_\sigma
    \leq 2 r_0^{d-1} |J(r_0)| \| \partial u \|_{L^2(A)}^2.
  \end{equation}

  Inserting \eqref{eq:trace:20} into \eqref{eq:trace:10}, multiplying
  the resulting inequality by $s^{d-1}$ and integrating over $s \in
  (r_0, r_1)$ yields
  \begin{align*}
    \smfrac{r_1^d - r_0^d}{d} \|u\|_{L^2(r_0\Sigma)}^2
    =~& \int_{s =
      r_0}^{r_1} s^{d-1} \int_{r_0\Sigma} |u|^2 dS\, ds \\
    \leq~&  2 r_0^{d-1} \int_{s = r_0}^{r_1} s^{d-1} \int_\Sigma
    \big|u(s\sigma)\big|^2 dS_\sigma\, ds
    + 2 r_0^{d-1} J(r_0) \smfrac{r_1^d - r_0^d}{d} \| \partial u
    \|_{L^2(A)}^2.
 \end{align*}
 Dividing through by $\frac{r_1^d - r_0^d}{d}$ we obtain
 \begin{displaymath}
   \|u\|_{L^2(r_0\Sigma)}^2 \leq \smfrac{2 d r_0^{d-1}}{r_1^d - r_0^d}
   \| u \|_{L^2(A)}^2 + 2 r_0^{d-1} J(r_0) \| \partial u
    \|_{L^2(A)}^2.
 \end{displaymath}
 Finally, estimating $r_0^d - r_1^d \geq (r_1 - r_0) r_1^{d-1}$ yields
 the stated trace inequality.
\end{proof}

\section{Acknowledgments}
We appreciate helpful discussions with Brian Van Koten.

\bibliographystyle{abbrv}
\bibliography{BQCFstab}

\end{document}